\PassOptionsToPackage{usenames,dvipsnames}{xcolor}

\documentclass[final]{siamart220329}

\ifpdf
  \DeclareGraphicsExtensions{.eps,.pdf,.png,.jpg}
\else
  \DeclareGraphicsExtensions{.eps}
\fi

\crefname{figure}{Fig.}{Figs.}
\Crefname{figure}{Fig.}{Figs.}

\usepackage[utf8]{inputenc}
\usepackage[T1]{fontenc}

\newsiamremark{remark}{Remark}
\newsiamremark{example}{Example}

\usepackage{graphicx}
\usepackage{hyphenat}
\usepackage{enumitem}
\usepackage[font={footnotesize,sf}]{caption}
\usepackage{subcaption}
\usepackage{algorithm,algpseudocode}
\usepackage{tikz}
  \usetikzlibrary{shapes.geometric,arrows,spy}
\usepackage{pgfplots}
  \pgfplotsset{compat=1.17}
  \usepgfplotslibrary{groupplots,external}
\usepackage{mathtools,amssymb}

\usepackage{nicematrix}
  \NiceMatrixOptions{cell-space-top-limit = 1pt,cell-space-bottom-limit = 1pt}

\DeclarePairedDelimiter\norm{\lVert}{\rVert}
\usepackage{cite,bookmark}

\headers{Legendre-Moment Transform for Control and Computation}
{X. Ning, G. Cheng, W. Zhang, and J.-S. Li}

\title{Legendre-Moment Transform for Linear Ensemble Control and Computation
\thanks{X.~Ning, G.~Cheng, and J.-S.~Li contributed equally to this work.}
}

\author{Xin Ning\thanks{Department of Electrical and Systems Engineering, Washington University in St.\ {}Louis, St.\ {}Louis, MO 63130. (\email{xin.ning@wustl.edu}, \email{wei.zhang@wustl.edu})}
\and Gong Cheng\thanks{School of Mathematical Sciences, Tongji University \& Key Laboratory of Intelligent Computing and Applications (Tongji University), Ministry of Education, Shanghai 200092, China. This work was conducted when this author was with the Department of Electrical and Systems Engineering at Washington University in St.\ {}Louis. (\email{gongch@tongji.edu.cn})}
\and Wei Zhang\footnotemark[2]
\and Jr-Shin Li\thanks{Corresponding author. Department of Electrical and Systems Engineering, Washington University in St.\ {}Louis, St.\ {}Louis, MO 63130.} (\email{jsli@wustl.edu})}

\ifpdf
  \hypersetup{
    pdftitle={A Dynamic Legendre-Moment Approach to Control of Linear Ensemble Systems},
    pdfauthor={X. Ning, G. Cheng, W. Zhang, and J.-S. Li}
  }
\fi

\begin{document}

\maketitle


\begin{abstract}
  Ensemble systems, pervasive in diverse scientific and engineering domains, pose challenges to existing control methods due to their massive scale and underactuated nature. This paper presents a dynamic moment approach to addressing theoretical and computational challenges in systems-theoretic analysis and control design for linear ensemble systems. We introduce the Legendre-moments and Legendre-moment transform, which maps an ensemble system defined on the $L^2$-space to a Legendre-moment system defined on the $\ell^2$-space. We show that this pair of systems is of one-to-one correspondence and shares the same controllability property. This equivalence admits the control of an ensemble system through the control of the corresponding Legendre-moment system and inspires a unified control design scheme for linear ensemble systems using structured truncated moment systems. In particular, we develop a sampling-free ensemble control design algorithm, then conduct error analysis for control design using truncated moment systems and derive error bounds with respect to the truncation orders, which are illustrated with numerical examples.
\end{abstract}



\section{Introduction}

The task of simultaneous manipulation of a collection of structurally similar systems using a single source of broadcast controllers, known as \emph{ensemble control}, has gained increasing attention across various scientific domains from quantum control \cite{van2017robust,Li2006control,Jiang2012multiple,li2011optimal,chen2014sampling} and neuroscience \cite{Rosenblum04,li2013control,Zlotnik2012optimal,ching_13_control,Kafashan2015optimal,Li_NatureComm16} to robotics \cite{becker2012approximate,mather2014synthesis}. Its prevalence in diverse disciplines has sparked substantial interest in systems science and control in the past decades, while challenges remain persistent in both theoretical and computational aspects. The fundamental bottleneck of controlling ensemble systems arises from the underactuated nature, where in practice broadcast control inputs are  engineered to orchestrate the collective behavior of a sizable population or, in the limit, a continuum of systems as control can be only be applied at the population level. 

Extensive mathematical tools from diverse areas, including algebraic methods, e.g., polynomial approximation \cite{li2016ensemble} and separating points \cite{li2020separating}, statistical approaches, e.g., the method of moments \cite{zeng2016moment,Dirr16}, functional analysis \cite{helmke2014uniform,Dirr2021uniform}, and representation theory \cite{Chen2019structure}, have been proposed to conduct systems-theoretic analysis of ensemble systems. Many of the developed methods, however, focused on analyzing fundamental properties of ensemble systems but were not tailored or suitable for computational control design and numerical analysis for ensemble systems. For example, many existing ensemble control design strategies are sampling-based relying on utilizing a dense sample of systems to represent the entire ensemble. This leads to low computational efficiency with increase of the sampling size and unsatisfactory performance due to ill-posedness issues in numerical computations \cite{Chen11}. Therefore,  developing a unified paradigm to facilitate tractable analysis, computation, and control design for ensemble systems is imperative.

In this paper, we adopt the notion of ``moments'' used in functional analysis for identifying a function using an infinite-sequence and introduce \emph{ensemble moments} associated with an ensemble system defined on a Hilbert space. Specifically, we introduce dynamic  \emph{Legendre-moments} as building blocks to establish a moment-based framework that enables and facilitates systems-theoretic analysis and control design for linear ensemble systems defined on a Hilbert space. We derive the Legendre-moment system associated with an ensemble system and show that this pair of infinite-dimensional systems is of one-to-one correspondence and shares the same controllability property. This equivalence allows us to formulate a unified ensemble control design scheme based on the use of truncated moment systems, by which we develop a \emph{sampling-free} control design algorithm for linear ensemble systems. We observe a banded structure in the system matrix of truncated moment systems. We then exploit this structure to conduct convergence and error analyses of truncated moment systems to the moment system with respect to the truncation order. 

The organization of this paper is as follows. In \Cref{sec:Problem_Transformation}, we introduce 
Legendre\hyp{}moments and Legendre-moment systems associated with ensemble systems. We then establish the dynamic equivalence between linear ensemble systems and their Legendre\hyp{}moment counterparts. \Cref{sec:control.design} details the unified control design scheme for linear ensemble systems based on truncated moment systems, where we provide 
the convergence property relative to the truncation order. In \Cref{sec:error.analysis.and.sample-free}, a thorough error analysis is conducted, leveraging the banded structure of the system matrix governing moment dynamics. We introduce a sampling-free ensemble control design strategy and substantiate our findings through numerical examples, showcasing the efficacy of the proposed moment-based control design approach and validating the performed error analysis.

\section{Moment Dynamics of Linear Ensemble Systems}
\label{sec:Problem_Transformation}

In this paper, we consider the ensemble of linear systems parameterized by $\beta$ taking values on a compact interval $K$, i.e., $\beta\in K\subset\mathbb{R}$, given by
\begin{equation}
  \label{eq:linear_ensemble}
	\frac{\mathrm{d}}{\mathrm{d}t}x(t,\beta) = A(t,\beta)x(t,\beta) + B(t,\beta)u(t),
\end{equation}
where $x(t,\cdot)\in L^2(K,\mathbb{R}^n)$ is the system state, $u\in L^\infty([0,T], \mathbb{R}^m)$ is the control input, and $A(t,\cdot)\in  L^{\infty}(K,\mathbb{R}^{n\times n})$ and $B(t,\cdot)\in L^2(K,\mathbb{R}^{n\times m})$ are the system and control matrices, respectively. Without loss of generality, we take $K=[-1,1]$, as any compact interval in $\mathbb{R}$ is homeomorphic to $[-1,1]$.
 A typical task in ensemble control is to engineer and apply an open-loop, broadcast controller $u(t)$ to navigate the ensemble system as in \cref{eq:linear_ensemble} from an initial profile $x_0(\cdot)$ to the target profile $x_F(\cdot)$ in an approximate sense. This pertains to the controllability property of the entire ensemble.

\begin{definition}[Ensemble Controllability]
  \label{def:ensemble.controllability}
  Let $M\subseteq\mathbb{R}^n$ be a smooth manifold, $K\subset\mathbb{R}$ be a compact set, and $\mathcal{F}(K)$ be the space of $M$-valued functions endowed with a metric $\mathrm{d}(\cdot, \cdot)$. An ensemble of systems defined on a $M$ parameterized by $\beta\in K$, given by
  \begin{equation}
    \label{eq:general.ensemble.system}
    \frac{\mathrm{d}}{\mathrm{d}t}x(t,\beta) = F\bigl(t,\beta,x(t,\beta),u(t)\bigr),\quad x(t,\cdot)\in\mathcal{F}(K),
  \end{equation}
  is \emph{ensemble controllable} on $\mathcal{F}(K)$, if for any $\epsilon>0$ and initial profile of states $x(0,\cdot)=x_0(\cdot)\in\mathcal{F}(K)$, there exists a piecewise constant control input $u(t)$ that steers the system into an $\epsilon$-neighborhood of a target profile $x_F\in\mathcal{F}(K)$ at a finite time $T>0$, i.e., $\mathrm{d}\bigl(x(T,\cdot), x_F(\cdot)\bigr)<\epsilon$. 
\end{definition}

Throughout this paper, we consider $L^2$-ensemble controllability of \cref{eq:linear_ensemble}, in which case $\mathcal{F}(K)=L^2(K,\mathbb{R}^n)$ equipped with the metric induced by the standard $L^2$-norm. Since $L^2(K,\mathbb{R}^n)$ is a separable Hilbert space, any ensemble state has the form of an infinite sum $x(t,\beta)=\sum m_{k}(t)\phi_{k}(\beta)$ for a given countable basis $\{\phi_k(\beta)\}$. Leveraging this representation, we show that, by choosing a proper orthonormal basis of $L^2(K,\mathbb{R}^n)$, the linear ensemble in~\cref{eq:linear_ensemble} can be transformed into an infinite-dimensional system represented by ``moments''.

\subsection{Moment Dynamics under the Legendre Basis}
\label{subsec:Moment_using_Legendre}

In this section, we consider the \emph{generalized moments} with respect to the \emph{normalized} Legendre polynomials defined by
\begin{equation}
	\label{eq:Legendre_moments}
	m_k(t)=\int_{K} P_k(\beta)x(t,\beta)\,\mathrm{d}\beta,
\end{equation}
and we call $m_k(t)$ the $k$-th Legendre moment of the ensemble system $x(t,\beta)$. Because the $K$ is bounded and $x(t,\cdot)$ is differentiable in $t$, the moments $m_k(t)$ are differentiable in $t$. The normalized Legendre polynomials $P_0(\beta)=\tfrac{1}{\sqrt{2}}$, $P_1(\beta)=\sqrt{\frac{3}{2}}\beta$, \dots{} in \cref{eq:Legendre_moments} satisfy the following orthonormality and recurrence relations: for any $i,j\in\mathbb{N}$ and $k\in\mathbb{Z}_{+}$,
\begin{align}
  \langle P_i, P_j\rangle &= \int_K P_i(\beta) P_j(\beta)\,\mathrm{d}\beta = \delta_{ij}, \label{eq:orthonormality} \\
  c_{k} P_{k+1}(\beta) &= \beta P_k(\beta) - c_{k-1} P_{k-1}(\beta), \label{eq:recurrence}
\end{align}
where $c_k=\tfrac{k+1}{\sqrt{(2k+1)(2k+3)}}$. 
Since the set $\{P_k\}$ constitutes an orthonormal basis of $L^2(K,\mathbb{R})$, by \cref{eq:Legendre_moments} each state variable $x(t, \cdot) = \bigl(x_1(t,\cdot), \ldots, x_n(t,\cdot)\bigr)' \in L^2(K,\mathbb{R}^n)$ can be represented as a Fourier-Legendre series: 
\begin{align}
  \label{eq:Legendre.expansion}
  x_i(t,\cdot)= \sum_{k=0}^{\infty} m_{ik}(t)P_k(\cdot),
\end{align}
where $m_{ik}(t)$ denotes the $i$\textsuperscript{th} element of $m_k(t) = (m_{1k}(t),\ldots,m_{nk}(t))^\prime$.
By the Parseval's identity, the infinite sequence (considered as an infinite \emph{column} vector throughout the paper), $m(t)\doteq{}(m'_0(t), m'_1(t), \ldots)'$ 
is square summable, for all $t\geqslant 0$. The Fourier-Legendre series~\cref{eq:Legendre.expansion} then induces an isometric isomorphism between $x(t,\cdot)\in L^2(K,\mathbb{R}^n)$ and $m(t)\in \ell^2(\mathbb{R}^n)$. Through this isomorphism, we can transform the linear ensemble~\cref{eq:linear_ensemble} in $L^2$ to the associated \emph{moment system}, which is an infinite\hyp{dimensional} linear system in $\ell^2(\mathbb{R}^n)$, so that ensemble controllability of the original linear ensemble is equivalent to approximate controllability of the moment system. To show this equivalence, below we provide a ``prototype'' example that demonstrates the correspondence between the ensemble in $L^2$ and its moment system in $\ell^2$.

\begin{example}
  \label{ex:prototype}
  Consider a linear ensemble system defined on $L^2([-1,1],\mathbb{R}^n)$, in which $A(t,\beta)$ is linear in $\beta$ and $B(t,\beta)$ is constant, given by
  \begin{equation}
    \label{eq:prototype}
   \frac{\mathrm{d}}{\mathrm{d}t}x(t,\beta) =\beta Ax(t,\beta)+Bu.
  \end{equation}
  Using the Fourier-Legendre series in \cref{eq:Legendre.expansion} and the recurrence relation in \cref{eq:recurrence}, we obtain the moment dynamics associated with the ensemble system in \cref{eq:linear_ensemble}, for any $k\in\mathbb{N}$,
  \begin{equation}
    \label{eq:moment.dynamics}
    \dot{m}_k(t) = c_{k-1}Am_{k-1}(t)+c_{k}Am_{k+1}(t)+\sqrt{2}\delta_{0k} Bu(t),
  \end{equation}
  where $m_{-1}$ and $c_{-1}$ are treated as $0$, and $\delta_{0k}$ is the Kronecker delta function. Since the coefficients $\{c_k\}$ are bounded (by $1/\sqrt{3}$), from \cref{eq:moment.dynamics} we see that the infinite sequence $\dot{m}(t)=(\dot{m}_0(t), \dot{m}_1(t), \ldots)$ is square summable, i.e., $\dot{m}(t)\in \ell^2(\mathbb{R}^n)$, for all $t\geqslant0$. So if we write $m(t)$ and $\dot{m}(t)$ as (infinite) column vectors, then \cref{eq:moment.dynamics} becomes a linear system in $\ell{}^2(\mathbb{R}^n)$ of the form,
  \begin{equation}
    \label{eq:moment.dynamics.prototype.in.example}
    \dot{m}(t) = \hat{A}m(t)+\hat{B}u(t),
  \end{equation}
  where $\hat{A}$ and $\hat{B}$ have the following matrix forms,
  \[
    \hat{A}=\begin{pmatrix}
      0 & \frac{1}{\sqrt{3}}  & 0  & \\
      \frac{1}{\sqrt{3}} & 0 & \frac{2}{\sqrt{15}} & \\
      0 & \frac{2}{\sqrt{15}} & 0 & \ddots \\
      & & \ddots & \ddots &
    \end{pmatrix} \otimes A \ \text{ and }\ %
    \hat{B}=\begin{pmatrix} \sqrt{2} \\ 0 \\ 0 \\ \vdots \end{pmatrix} \otimes B.
  \]
  where $\otimes$ denotes the Kronecker product of matrices.
\end{example}

We note that both $\hat{A}$ and $\hat{B}$ are bounded operators, and obviously $\hat{B}$ is of finite rank. Also, for any $m = (m_0,m_1,\ldots)'\in\ell^2(\mathbb{R}^n)$, we have by \cref{eq:moment.dynamics.prototype.in.example} that
\begin{multline}
  \label{eq:A.hat.boundedness.prototype}
  \norm[\big]{\hat{A}m}_{\ell^2} =
  \norm[\big]{\begin{pNiceMatrix} 0 \\ \frac{1}{\sqrt{3}} Am_0 \\ \frac{ 2}{\sqrt{15}} Am_1 \\ \vdots \end{pNiceMatrix} +
  \begin{pNiceMatrix} \frac{1}{\sqrt{3}}Am_1 \\ \frac{2}{\sqrt{15}}Am_2 \\ \frac{3}{\sqrt{35}}Am_3 \\ \vdots \end{pNiceMatrix}}_{\ell^2} \\
  \leqslant \norm{\sharp(\{Am_i\})}_{\ell^2} + \norm{\flat(\{Am_i\})}_{\ell^2} \leqslant 2 \norm{A}_{2}\cdot \norm{m}_{\ell^2},
\end{multline}
where $\norm{A}_2$ is the matrix 2-norm of $A$, and $\sharp(\cdot)$ and $\flat(\cdot)$ denote the forward and backward shift operators on $\ell^2(\mathbb{R}^n)$, respectively (see Appendix). Therefore, we conclude that the linear ensemble~\cref{eq:prototype} in $L^2(K,\mathbb{R}^n)$ induces the moment system in \cref{eq:moment.dynamics.prototype.in.example}, which is well-defined and evolves in $\ell^2(\mathbb{R}^n)$. 

\begin{remark}
  \label{rem:long.remark}
  (i) Notice that the ensemble system in \cref{eq:linear_ensemble} is an infinite-dimensional time-varying linear system. Its moment system is also a time-varying linear system due to the linearity of the moment transformation defined in \cref{eq:Legendre_moments}. 
  
  (ii) The matrix $\hat{A}$ in \cref{eq:moment.dynamics.prototype.in.example} belongs to a class of operators known as the \emph{banded} matrix, whose definition and properties are summarized in the appendix. The band structure plays a central role in both the development of the infinite\hyp{}dimensional moment system in $\ell^2(\mathbb{R}^n)$ and its applications.

  (iii) Since the subsystem $x(t,\beta)$ of the ensemble for each $\beta$ evolves in $\mathbb{R}^n$, we use the notation $\ell^2(\mathbb{R}^n)$ to denote the space consisting of square\hyp{}summable sequences of $\mathbb{R}^n$\hyp{}vectors. It is trivially isometrically isomorphic to $\ell^2$, the Hilbert space of square\hyp{}summable scalar sequences. The only difference between the two is that the forward/backward shift operator in $\ell^2(\mathbb{R}^n)$ shifts a corresponding $\ell^2$\hyp{}sequence $n$ times, as we see in \cref{eq:A.hat.boundedness.prototype}.
\end{remark}

\subsection{Duality of Linear Ensemble Systems} \label{subsec:equivalence.of.the.two.systems}

The introduced Legendre moment transform maps a linear ensemble system to a linear control system described by Legendre moments. This inspires the examination of equivalence between this pair of infinite-dimensional linear systems. In this section, we reveal duality between a linear ensemble and its Legendre-moment system in terms of their systems properties, e.g., controllability. 

The moment system in \cref{eq:moment.dynamics.prototype.in.example} is an infinite-dimensional linear system in the Hilbert space $\ell^2(\mathbb{R}^n)$. To reveal its controllability resemblance to the linear ensemble in \cref{eq:prototype}, we define the \emph{approximate controllability} for linear systems in a Banach space as follows.

\begin{definition}[Approximate Controllability~\cite{triggiani1975controllability}]
  \label{def:approximate_controllability}
  Consider a linear system evolving on the Banach space $\mathcal{X}$ of the form $ \dot{m}(t)= \hat{A}m(t)+\hat{B}u(t)$, where $\hat{A}\in \mathcal{B}(\mathcal{X})$ and $\hat{B}\in\mathcal{B}(\mathbb{R}^m,\mathcal{X})$ are bounded operators, and the control input $u(t)\in L^\infty([0,T],\mathbb{R}^m)$ is bounded. We call this system \emph{approximately controllable} on interval $[0,T]$ if given any $\epsilon>0$ and two arbitrary initial and final points $x_0$ and $x_F$,respectively, in $\mathcal{X}$, there is an admissible $u(t)$ on $[0, T]$ steering $x_0$ to an $\epsilon$-neighborhood of $x_F$.
\end{definition}

Given the correspondence between the linear ensemble and the moment system presented in \cref{ex:prototype}, using the Parseval's identity, we build the following equivalence of controllability for the two systems.

\begin{theorem}
  \label{prop:controllability.equiv.linear.case}
  The linear ensemble in \cref{eq:prototype} is $L^2$-ensemble controllable if and only if its associated Legendre-moment system in \cref{eq:moment.dynamics.prototype.in.example} is approximately controllable in $\ell^2(\mathbb{R}^n)$.
\end{theorem}

\begin{proof}
  (Sufficiency): We have shown in the previous example that if the state $x(t,\cdot)$ in $L^2(K,\mathbb{R}^n)$ satisfies \cref{eq:prototype}, then its moments must satisfy the moment dynamics~\cref{eq:moment.dynamics} (and, equivalently, \cref{eq:moment.dynamics.prototype.in.example}). On the other hand, if an infinite sequence $m(t)=(m_0(t), m_1(t), \ldots)$ in $\ell^2(\mathbb{R}^n)$ satisfies the dynamics in \cref{eq:moment.dynamics}, then it corresponds to an $L^2$-function $x(t,\cdot)=\sum_{k=0}^{\infty} m_k(t)P_k(\cdot)$ which satisfies \cref{eq:prototype}:
  \begingroup
  \allowdisplaybreaks
  \begin{align*}
    \dot{x}(t,\beta) &=\sum_{k=0}^\infty \dot{m}_k(t)P_k(\beta) =\sum_{k=0}^{\infty} \bigl[c_{k-1}Am_{k-1}(t)P_k(\beta) \\ &\qquad\qquad\qquad\qquad\qquad\quad +c_kAm_{k+1}(t)P_k(\beta)\bigr]+\sqrt{2}Bu(t)P_0(\beta) \\
    &=\sum_{k=0}^\infty \bigl[c_kP_{k+1}(\beta) + c_{k-1}(t)P_{k-1}(\beta)\bigr]Am_k(t)  +Bu(t) \\
    &=\sum_{k=0}^{\infty} \beta Am_k(t)P_k(\beta) +Bu(t) = \beta Ax(t,\beta)+Bu.
  \end{align*}
  \endgroup
  
  Now let us suppose the linear ensemble in \cref{eq:prototype} is $L^2$\hyp{}ensemble controllable, then for any $x_F\in L^2(K,\mathbb{R}^n)$, we can find an admissible control input $u(t)$ that steers $x(t,\cdot)$ to an $\epsilon$-neighborhood of $x_F$. By the isometry between $L^2$ and $\ell^2$:
  \begin{equation}
    \label{eq:isometry}
      \norm{x(T,\cdot)-x_F}_{L^2}=\norm{m(T)-m_F}_{\ell^2},
  \end{equation}
  the same control input $u(t)$ steers $m(t)$ to an $\epsilon$\hyp{}neighborhood of $m_F$ within the same time frame. In other words, the \emph{reachable set} of \cref{eq:moment.dynamics.prototype.in.example} is dense in $\ell^2$, which shows approximate controllability. This concludes the sufficiency part of the proof. 
  
  (Necessity:) This is obvious by reversing the process above.
\end{proof}

In light of the moment dynamics in \cref{ex:prototype}, we can generalize the correspondence between linear ensembles in $L^2(K,\mathbb{R}^n)$ and moment systems in $\ell^2(\mathbb{R}^n)$ to a broader class of ensemble systems. In particular, we extend this property to the ensemble system governed by vector fields in polynomial forms, i.e., $A(t,\beta)$ and $B(t,\beta)$ are polynomials in the system parameter $\beta$. 

\begin{proposition}
  \label{thm:moment.dynamics.for.polynomial.in.beta}
  Consider the linear ensemble system in \cref{eq:linear_ensemble} with the matrices $A(t,\beta)$ and $B(t,\beta)$ being polynomials in $\beta$, i.e.,
  \begin{equation}
    \label{eq:A.and.B.polynomials.in.beta}
    A(t,\beta)=\sum_{i=0}^{N_1} P_i(\beta)A_i(t), \quad
    B(t,\beta)=\sum_{j=0}^{N_2} P_j(\beta)B_j(t),
  \end{equation}
  then its moment system is linear and is of the form
  \begin{equation}
    \label{eq:moment.system.banded.system.matrix}
    \dot{m}(t) = \hat{A}(t)m(t) + \hat{B}(t)u(t),
  \end{equation}
  where $\hat{A}$ is an infinite $2nN_1$\hyp{}banded matrix as in \cref{def:banded.matrix.and.bandwidth} in \cref{appendix:infinite.banded.matrix} and $\hat{B}$ is of rank $nN_2$, i.e.
  \begin{subequations}
    \label{eq:A.hat.and.B.hat}
    \begin{flalign}
      \hat{A}(t) &=
      \begin{pNiceArray}{cccccc}
        A_{00}(t) & \cdots & A_{0{N_1}}(t) & & \Block{2-2}<\large>{\mathbf{0}} & \\
        \vdots & \Block{3-3}<\large>{\ddots} & & \ddots & & \\
        A_{{N_1}0}(t) & & & & & \\
        \Block{2-2}<\large>{\mathbf{0}} & \ddots & & & & \\
        & & & & &
      \end{pNiceArray}, \label{eq:polynomials.in.beta.A.hat} \\
      \hat{B} &=\begin{pmatrix} B_{0}(t)' & B_{1}(t)' & \ldots & B_{N_2}(t)' & 0  & \ldots \end{pmatrix}'. \label{eq:polynomials.in.beta.B.hat}
    \end{flalign}
  \end{subequations}
\end{proposition}

\begin{proof}
  By differentiating \cref{eq:Legendre_moments}, we obtain
  \begingroup
  \allowdisplaybreaks
  \begin{align}
    \dot{m}_k(t) &= \int_{K} P_k(\beta)\dot{x}(t,\beta)\,\mathrm{d}\beta = \int_{K} P_k(\beta)\bigl[A(t,\beta)x(t,\beta)+B(t,\beta)u(t)\bigr] \,\mathrm{d}\beta \nonumber \\
    &= \sum_{i=0}^{N_1} A_i(t) \int_{K} P_k(\beta)P_{i}(\beta)x(t,\beta)\,\mathrm{d}\beta  +\sum_{j=0}^{N_2} B_j(t)u(t) \int_{K} P_k(\beta)P_{j}(\beta)\,\mathrm{d}\beta, \nonumber \\
    \begin{split}
      &= \sum_{i=0}^{N_1} A_i(t) \int_{K} P_k(\beta)P_{i}(\beta)x(t,\beta)\,\mathrm{d}\beta + B_k(t)u(t),
    \end{split}
    \label{eq:m_k.derivative}
  \end{align}
  \endgroup
  where $h_k\doteq{}0$ for $k>N_2$. It is now evident that \cref{eq:polynomials.in.beta.B.hat} holds for $\hat{B}$, in which only the first $N_2$ entry blocks are nonzero. To show that $\hat{A}$ has the band structure in \cref{eq:polynomials.in.beta.A.hat}, we note that the product of two Legendre polynomials can be expressed as a linear combination \cite{al-salam1957product}
  \begin{equation}
    \label{eq:Legendre.linear.combination}
    P_k(\beta)P_i(\beta)=\mathop{\smash{\sum_{r=0}^{i}}} d_{kir}P_{k+i-2r}(\beta),
  \end{equation}
  and the coefficients $d_{kir}$ have the following form
  \[
    d_{kir}=\frac{\sqrt{(2k+2i-4r+1)(2k+1)(2i+1)}}{\sqrt{2}(2k+2i-2r+1)}\cdot \frac{\mathcal{G}_r \mathcal{G}_{k-r} \mathcal{G}_{i-r}}{\mathcal{G}_{k+i-r}},
  \]
  where
  \[
    \mathcal{G}_{s}\doteq{}\frac{1}{2} \Bigl(1+\frac{1}{2}\Bigr) \cdots \Bigl(s-1+\frac{1}{2}\Bigr)\big/s!=\frac{\Gamma(s+\frac{1}{2})}{\Gamma(s+1)\Gamma(\frac{1}{2})}
  \]
  for $s\in \mathbb{N}$. So we can write the sum in \cref{eq:m_k.derivative} as
  \begingroup
  \allowdisplaybreaks
  \begin{align}
    \sum_{i=0}^{N_1} A_i(t)\int_{K} P_k(\beta)P_{i}(\beta)x(t,\beta)\,\mathrm{d}\beta &= \sum_{i=0}^{N_1} A_i(t)\int_{K} \sum_{r=0}^{i} d_{kir}P_{k+i-2r}(\beta)x(t,\beta)\,\mathrm{d}\beta \nonumber \\
    &= \sum_{i=0}^{N_1}\sum_{r=0}^{i} d_{kir} A_i(t) m_{k+i-2r}(t), \label{eq:A.hat.entry.coefficients}
  \end{align}
  \endgroup
  where $d_{kir}$ and $m_{k+i-2r}$ with negative indices are treated as $0$. It is clear the sum in \cref{eq:A.hat.entry.coefficients} is a linear combination of the moments: \(m_{k-N_1}, m_{k-N_1+1}, \ldots, m_k, \ldots, m_{k+N_1}\), so when represented as an (infinite) matrix, $\hat{A}$ is banded with a bandwidth of $2nN_1$, and each of its entry blocks has the closed form: $A_{kl}(t) = \sum_{i=|k-l|}^{N_1} d_{ki\lfloor\tfrac{i+k-l}{2}\rfloor} A_i(t)$.
  
  Next, we show that $\hat{A}$ and $\hat{B}$ are bounded operators. This is obvious for $\hat{B}$, since we know from \cref{eq:m_k.derivative} that $\norm{\hat{B}}^2\leqslant \sum_{i=0}^{N_2}\norm{B_i}^2$.
  To show that $\hat{A}$ is also bounded, by \cref{lem:bounded.banded.matrix} in the appendix, it suffices to prove that the entries of $\hat{A}$ are bounded. To see this, we first observe that, for all $i,r\in [0:N_1]\doteq{}\{0,1,2,\ldots, N_1\}$,
  \[
    \lim_{k\to \infty} \frac{\mathcal{G}_{k-r}}{\mathcal{G}_{k+i-r}}=\lim_{k\to \infty}\frac{\Gamma(k-r+\frac{1}{2})}{\Gamma(k-r+1)}\cdot \frac{\Gamma(k-r+i+1)}{\Gamma(k-r+i+\frac{1}{2})}=1.
  \]
  Also note that $0<\mathcal{G}_{r}, \mathcal{G}_{i-r}\leqslant 1$, then this gives
  \[
    \lim_{k\to\infty} d_{kir}\leqslant \sqrt{N_1+1}.
  \]
  In other words, $d_{kir}$ are uniformly bounded for all $k\in \mathbb{N}$ and $i,r\in[0:N_1]$. It then follows immediately from \cref{eq:A.hat.entry.coefficients} that the coefficient of $m_{k+s}$, where $s\in[-N_1:N_1]$, are uniformly bounded, for all $k$ and $s$. Equivalently, the entries of $\hat{A}$ are bounded, which concludes our proof.
\end{proof}

\Cref{thm:moment.dynamics.for.polynomial.in.beta} extends the scope of our analysis, so that the correspondence between the linear ensemble (consisting of uncountably-many subsystems) and the moment system (countably-many moment terms) is applied to linear ensembles with polynomial coefficients. Following exactly the proof of \cref{prop:controllability.equiv.linear.case}, we conclude that ensemble controllability of a time-invariant linear ensemble with $A$ and $B$ in the form of \cref{eq:A.and.B.polynomials.in.beta} is equivalent to approximate controllability of the infinite\hyp{}dimensional system~\cref{eq:moment.dynamics.prototype.in.example}, where $\hat{A}$ and $\hat{B}$ take the form in \cref{eq:A.hat.and.B.hat}. This result is summarized in the following theorem. 

\begin{theorem}
  \label{thm:controllability.equiv.polynomial.case}
  The linear ensemble in \cref{eq:linear_ensemble} with $A(\beta)$ and $B(\beta)$ being the polynomials in \cref{eq:A.and.B.polynomials.in.beta} is $L^2$-ensemble controllable if and only if its associated Legendre moment system in \cref{eq:moment.dynamics.prototype.in.example} is approximately controllable in $\ell^2(\mathbb{R}^n)$.
\end{theorem}

The equivalence presented in \cref{thm:controllability.equiv.polynomial.case} establishes \emph{duality} between linear ensemble systems and their Legendre-moment systems, which extends the existing results presented in \cite{zeng2016moment}. This duality property provides a new approach to verifying ensemble controllability through approximate controllability of the associated Legendre-moment system. We demonstrate this by using infinite-dimensional LTI systems, which exhibit explicit algebraic characterization.

\begin{lemma}
  \label{lem:approx.controllability}
  Consider the infinite-dimensional LTI system evolving on the Hilbert space $\ell^2(\mathbb{R}^n)$ of the form $\dot{m}(t)= \hat{A}m(t)+\hat{B}u(t)$, 
  where $\hat{A}\in \mathcal{B}(\ell^2(\mathbb{R}^n))$ and $\hat{B}\in\mathcal{B}(\mathbb{R}^m,\ell^2(\mathbb{R}^n))$ are bounded operators, and the control input $u(t)\in L^\infty([0,T],\mathbb{R}^m)$. This system is \emph{approximately controllable} on interval $[0,T]$ if and only if the subspace spanned by
  \begin{equation}
    \label{eq:CoAB}
    \{\hat{A}^n\hat{B}\}=\{\hat{b}_1,\ldots,\hat{b}_m, \hat{A}\hat{b}_1,\ldots,\hat{A}\hat{b}_m, \hat{A}^2\hat{b}_1,\ldots\}
  \end{equation}
  is dense in $\ell^{2}(\mathbb{R}^n)$, where $\hat{b}_1, \ldots, \hat{b}_m$ are the column vectors of the matrix $\hat{B}$.
\end{lemma}

\begin{proof}
  See Theorem~3.1.1 and Corollary~3.1.2 in \cite{triggiani1975controllability}. Note that while the original results are for systems in complex Banach spaces, the same proof works for real\hyp{}valued systems as well.
\end{proof}

Henceforth, we will use $\mathrm{Co}(\hat{A},\hat{B})$ to denote the subspace spanned by $\{\hat{A}^n\hat{B}\}$ defined in \cref{eq:CoAB} and call it the \emph{controllability subspace} of $\hat{A}$ and $\hat{B}$. We now illustrate that we can examine ensemble controllability of a linear ensemble system through the denseness of the controllability subspace, $\mathrm{Co}(\hat{A},\hat{B})$, generated by its moment system following \cref{lem:approx.controllability} and our discussion in \Cref{subsec:equivalence.of.the.two.systems}.

\begin{example}
  \label{ex:beta.controllable}
  Consider the scalar linear ensemble system defined on $L^2(K,\mathbb{R})$ controlled by a single input,
  \begin{equation}
    \label{eq:example.beta.ensemble.system.controllable}
    \frac{\mathrm{d}}{\mathrm{d}t}x(t,\beta)=\beta x(t,\beta)+u, \quad \beta\in [-1,1].
  \end{equation}
  Following the steps in \cref{ex:prototype}, we obtain the moment system as in \cref{eq:moment.dynamics.prototype.in.example} with 
  \[
    \hat{A}=\begin{pNiceMatrix}
      0 & \frac{1}{\sqrt{3}}  & 0  & \\
      \frac{1}{\sqrt{3}} & 0 & \frac{2}{\sqrt{15}} & \\
      0 & \frac{2}{\sqrt{15}} & 0 & \smash{\ddots} \\
      & & \smash{\ddots} & \smash{\ddots}
    \end{pNiceMatrix}\ %
    \text{and}\ %
    \hat{B}=\begin{pmatrix} \sqrt{2} & 0 & 0 & \dots \end{pmatrix}'.
  \]
  The subspace ${\rm Co}(\hat{A}, \hat{B})$ is spanned by
  \begin{equation}
    \label{eq:example.beta.controllable.moment.kalman}
    \Bigl\{
    \begin{pNiceMatrix}
      1 \\ 0 \\ 0 \\ 0 \\ 0 \\ 0 \\ \vdots 
    \end{pNiceMatrix},
    \begin{pNiceMatrix} 
      0 \\ \frac{1}{\sqrt{3}} \\ 0 \\ 0 \\ 0 \\ 0 \\ \vdots 
    \end{pNiceMatrix},
    \begin{pNiceMatrix} 
      \frac{1}{3} \\ 0 \\ \frac{2}{3\sqrt{5}} \\ 0 \\ 0 \\ 0 \\ \vdots 
    \end{pNiceMatrix},
    \begin{pNiceMatrix} 
      0 \\ \frac{\sqrt{3}}{5} \\ 0 \\ \frac{2}{5\sqrt{7}} \\ 0 \\ 0 \\ \vdots 
    \end{pNiceMatrix},
    \begin{pNiceMatrix} 
      \frac{1}{5} \\ 0 \\ \frac{4}{7\sqrt{5}} \\ 0 \\ \frac{8}{105} \\ 0 \\ \vdots 
    \end{pNiceMatrix}, \ldots\Bigr\},
  \end{equation}
  which is \cref{eq:CoAB} up to a factor $\sqrt{2}$.
  Due to the upper triangular structure of the vectors in \cref{eq:example.beta.controllable.moment.kalman}, each element in the canonical basis of $\ell{}^2$, that is, $\{(1,0,0,\ldots)', (0,1,0,\ldots)',$ $(0,0,1,\ldots)', \ldots\}$, is a finite linear combination of vectors in \cref{eq:example.beta.controllable.moment.kalman}, which implies that $\mathrm{Co}(\hat{A},\hat{B})$ is dense in $\ell^2$. Therefore, by \cref{lem:approx.controllability}, the moment system is approximately controllable, and, equivalently, the ensemble \cref{eq:example.beta.ensemble.system.controllable} is $L^2$\hyp{}controllable.
\end{example}

\begin{example}
  \label{ex:beta.oscillator.uncontrollable}
  Consider the harmonic oscillator ensemble defined on $L^2(K,\mathbb{R}^2)$ with the frequency $\beta$ varying on $[-1,1]$, given by,
  \begin{equation}
    \label{eq:oscillator}    
    \frac{\mathrm{d}}{\mathrm{d}t} x(t,\beta)=\beta
    \begin{pNiceMatrix}
      0 & 1 \\
      -1 & 0
    \end{pNiceMatrix} x(t,\beta)+
    \begin{pNiceMatrix} 1 \\ 0 \end{pNiceMatrix}u(t).
  \end{equation}
  According to \cref{prop:controllability.equiv.linear.case}, controllability of this ensemble system depends on that of its moment system as in \cref{eq:moment.dynamics.prototype.in.example}, where 
  \[
    \hat{A}=\begin{pNiceArray}{cc:cc:cc:c}
      \Block{2-2}<\Large>{\mathbf{0}} &  & 0 & \frac{1}{\sqrt{3}} & \Block{2-2}<\Large>{\mathbf{0}} & & \\
      &  & -\frac{1}{\sqrt{3}} & 0 & & & \\
      \hdottedline
      0 & \frac{1}{\sqrt{3}} & \Block{2-2}<\Large>{\mathbf{0}} & & 0 & \frac{2}{\sqrt{15}} & \\
      -\frac{1}{\sqrt{3}} & 0 & & & -\frac{2}{\sqrt{15}} & 0 & \\
      \hdottedline
      \Block{2-2}<\Large>{\mathbf{0}} & & 0 & \frac{2}{\sqrt{15}} & \Block{2-2}<\Large>{\mathbf{0}} & & \Block{2-1}<\large>{\ddots} \\
      & & -\frac{2}{\sqrt{15}} & 0 & & & \\
      \hdottedline
      & & & & \Block{2-2}<\large>{\ddots} & & \Block{2-1}<\large>{\ddots} \\
      & & & & & &
    \end{pNiceArray}
  \]
  and $\hat{B}=\begin{pmatrix} \sqrt{2} & 0 & 0 & \ldots \end{pmatrix}'$. We can then compute the controllable subspace $\mathrm{Co}(\hat{A},\hat{B})$ which is spanned by
  \[\biggl\{
    \begin{aligned}
      &(1,0,0,0,\ldots)', (0,0,0,\tfrac{\sqrt{3}}{3},0,0,\ldots)', \\
      &(\tfrac{1}{3},0,0,0,\tfrac{2}{3\sqrt{5}},0,\ldots)', (0,0,0,\ast,\ast,\ldots)', \ldots
    \end{aligned}
    \biggr\}.
  \]
  Obviously the vector $(0, 1, 0, \ldots)'$ belongs to the orthogonal complement of $\mathrm{Co}(\hat{A},\hat{B})$, and thus $\mathrm{Co}(\hat{A},\hat{B})$ is not dense in $\ell^2(\mathbb{R}^2)$. This suggests that the moment system is not approximately controllable, and in turn the system in \cref{eq:oscillator} is not $L^2$-ensemble controllable.
\end{example}

The established duality presented in this section naturally translates the control design for linear ensemble systems to their dual moment systems, which we will elaborate in the following sections.

\section{Ensemble Control Design via Moment Systems} \label{sec:control.design}

To study the infinite-dimensional moment system defined on $\ell^2(\mathbb{R}^n)$, it is natural to consider its truncated counterparts as finite-dimensional approximations. Unfortunately, conventional truncation methods often fall short in accurately representing the controllability properties of the original system. However, in this section, we address this limitation by exploiting the banded matrix structure in moment dynamics, demonstrating the approximation of the infinite-dimensional moment system of the form \cref{eq:moment.system.banded.system.matrix} through its finite truncations. This allows the seamless application of tools and theory for finite-dimensional linear systems to comprehend and manipulate the Legendre-moment system. The convergence of truncated systems inspires the development of a control scheme for the moments in \cref{eq:moment.system.banded.system.matrix}, transforming the control design for linear ensemble systems into a sequence of tasks for systems defined in finite-dimensional spaces. Furthermore, we establish a rigorous estimation of the error resulting from such finite approximations, a new contribution to the numerical application of ensemble control theory.

To set the stage, we first give a mathematical definition of the truncated systems and analyze its approximation properties. In what follows, we identify $\ell^2(\mathbb{R}^n)$ with $\ell^2$, since their difference is irrelevant to the truncation (see \cref{rem:long.remark}).

\subsection{Truncation of Infinite-Dimensional Systems} \label{subsec:truncation}
Let us consider a time-invariant infinite-dimensional linear system,
\begin{equation}
  \label{eq:time-inv.inf.linear.system}
  \dot{m}(t)=\hat{A}m(t)+\hat{B}u(t), \quad m(t)\in\ell^2,
\end{equation}
where $\hat{A}$ has the form of an infinite $b$\hyp{}banded matrix with uniformly bounded entries and $\hat{B}$ is a finite rank operator in the form of \cref{eq:polynomials.in.beta.B.hat}. For an admissible $u(t)$, we can write the solution $m(t)$ as (see \cite{triggiani1975controllability})
\begin{equation}
  \label{eq:m(t).diff.eq.solution.infinite.dim}
  m(t)=e^{t\hat{A}}m(0)+\int_{0}^{t} e^{(t-\tau)\hat{A}}\hat{B}u(\tau)\,\mathrm{d}\tau.
\end{equation}
Our goal is to approximate the system in \cref{eq:time-inv.inf.linear.system} and its solution~\cref{eq:m(t).diff.eq.solution.infinite.dim} using finite truncations. To this end, let us first define the truncated systems. Given an infinite sequence $\xi=(\xi_0, \xi_1, \ldots)'$ in $\ell^2$ (represented in the form of an infinite-dimensional column vector), a positive integer $N\in \mathbb{N}_{+}$, and a vector $\bar{\xi}=(\xi_0, \ldots, \xi_{N-1})'\in \mathbb{R}^{N}$, if $P_N$ and $\iota$ denote the projection and inclusion operators, respectively, i.e.,
\begingroup
\allowdisplaybreaks
\begin{align*}
  P_N &:\ell^2 \to\mathbb{R}^{N}, & P_N\xi &\doteq{}(\xi_0,\ldots,\xi_{N-1})^\prime; \\
  \iota &:\mathbb{R}^{N} \to\ell^2, & \iota\bar{\xi} &\doteq{}(\xi_0, \ldots, \xi_{N-1}, 0, \ldots)^\prime,
\end{align*}
\endgroup
we can then define the truncated operators $\hat{A}_{N}:\mathbb{R}^{N}\to \mathbb{R}^{N}$ and $\hat{B}_{N}:\mathbb{R}^m \to \mathbb{R}^{N}$ as $\hat{A}_{N}\doteq{}P_{N}\hat{A}\iota$ and $\hat{B}_{N}\doteq{}P_{N}\hat{B}$. \Cref{fig:truncation} is an illustration of the truncation of operators. Obviously, the truncations $\hat{A}_N$ are also $b$\hyp{}banded. By setting the initial condition $\bar{m}(0)\doteq{}P_{N}m(0)$ and the state $\bar{m}\in\mathbb{R}^{N}$, the truncated moment system of order $N$ evolving on $\mathbb{R}^{N}$ is defined by
\begin{equation}
  \label{eq:truncated.moment.dynamics}
  \dot{\bar{m}}(t) = \hat{A}_N\bar{m}(t)+\hat{B}_N \bar{u}(t), \quad \bar{m}(0)=P_{N}m(0).
\end{equation}

It is worth highlighting that the truncation is \emph{not} conducted on the moment state $m(t)$, but rather on the system dynamics, i.e., both $\hat{A}_N$ and $\hat{B}_N$ are truncated operators. Therefore, for each truncation order $N$, the truncated system~\cref{eq:truncated.moment.dynamics} follows different dynamics. In the next theorem, we compare the truncated finite-dimensional systems with the infinite-dimensional moment system, and show that the states $\bar{m}(t)$ converge to the infinite moment state $m(t)$ as $N$ increases.

\begin{figure}[htbp]
  \centering
  \begin{tikzpicture}[scale=0.75, line cap=round]
    \tikzstyle{textbox} = [rectangle, rounded corners, minimum width=25mm, minimum height=5mm, text centered, draw=black!50]
    \tikzstyle{arrow} = [->, >=stealth]
    \draw[semithick, loosely dash dot] (4.3,0.7) -- (0,5);
    \draw[semithick] (0.6,5) -- (4.3,1.3);
    \draw[semithick] (0,4.4) -- (3.7,0.7);
    \draw[semithick] (4.5,5) -- (3,5);
    \draw[semithick] (0,2) -- (0,0.5);
    \draw[thick, dotted, line cap=butt] (4.5,0.5) -- (4.7,0.3);
    \draw[thick, dotted, line cap=butt] (4.6,5) -- (4.9,5);
    \draw[thick, dotted, line cap=butt] (0,0.4) -- (0,0.1);
    \node at (2.4,0) {$\hat{A}$};
    \draw[densely dashed, red, thick] (3,5) -- (0,5) -- (0,2) -- (4.5,2);
    \draw[densely dashed, red, thick] (3,5) -- (3,0.5);
    \node[below, red] at (1.5,1.9) {$\hat{A}_{N}$};
    \node (txt) [textbox] at (5.7,3.6) {$\hat{A}_{N}=P_{N}\hat{A}\,\iota\in\mathbb{R}^{N\times N}$};
    \draw[arrow, very thin, draw=black!50] (txt) -- (5.7,3) -- (3.05,3);
  \end{tikzpicture}
  \caption{An illustration of finite truncation of a banded matrix. $\hat{A}_N$ is an $N$-by-$N$ square truncation of $\hat{A}$.}
  \label{fig:truncation}
\end{figure}
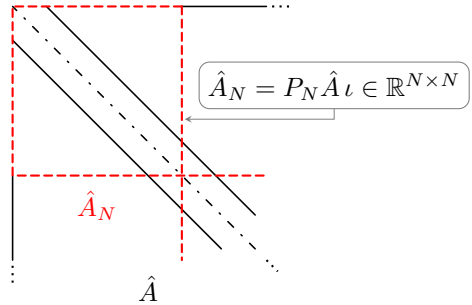

\begin{theorem}
  \label{thm:truncated.systems.approx}  
  Consider the moment system in \cref{eq:time-inv.inf.linear.system} where $\hat{A}$ is an infinite $b$\hyp{}banded matrix with uniformly bounded entries and $\hat{B}$ is a finite rank operator in the form of \cref{eq:polynomials.in.beta.B.hat}. Let $\hat{A}_{N}$ denote the truncated operator of $\hat{A}$, and $e^{t\hat{A}}$ and $e^{t\hat{A}_{N}}$ denote the semigroups of bounded linear operators generated by $\hat{A}$ and $\hat{A}_{N}$, respectively. We have that, for any $\xi\in\ell^2$, 
  \begin{equation}
    \label{eq:semigroup.convergence}
    (\iota{}e^{t\hat{A}_{N}}P_{N})\xi\to e^{t\hat{A}}\xi
  \end{equation}
  in $\ell^2$ as $N\to\infty$. Consequently, following the same control input $u(t)$, the moment $m(t)$ in \cref{eq:time-inv.inf.linear.system} and its truncation $\bar{m}(t)$ in \cref{eq:truncated.moment.dynamics} satisfy the approximation property, $\iota\bar{m}(t)-m(t)\to 0$, in $\ell^2$ as $N\to\infty$.
\end{theorem}

\begin{proof}
  Note that for the uniformly continuous semigroup (see \cref{def:semigroup}) $\iota e^{t\hat{A}_N}P_N$, $t\geqslant 0$, its infinitesimal generator is $\iota \hat{A}_NP_N\in\mathcal{B}(\ell^2)$. To show the convergence in \cref{eq:semigroup.convergence}, by \cref{thm:trotter-kato} in \cref{appendix:semigroup}, for bounded infinitesimal generators $\iota\hat{A}_NP_N$ and $\hat{A}$, we need to show that (i) $\iota\hat{A}_NP_N\xi - \hat{A}\xi\to 0$ for all $\xi\in\ell^2$ as $N\to\infty$, and (ii) there exists a $\lambda_0$ with $\mathrm{Re}\,\lambda_0>\omega$ for which $(\lambda_0I-\hat{A})$ is surjective. Since $\hat{A}$ is bounded, (ii) holds for all $\lambda_0 > \norm{\hat{A}}$. Next, we denote $\delta \doteq \hat{A}\xi - \iota \hat{A}_N P_N \xi$ and show that $\delta\rightarrow0$ as $N\rightarrow \infty$.

  Let us write $\hat{A}$ into a block matrix form $\hat{A} = \begin{psmallmatrix} \hat{A}_N & \hat{A}_{12} \\ \hat{A}_{21} & \hat{A}_{22} \end{psmallmatrix}$. Since $\hat{A}$, $\hat{A}_N$, and $\hat{A}_{22}$ are all $b$\hyp{}banded with the same uniform bound for their entries, say $M$, then by \cref{lem:bounded.banded.matrix} in \cref{appendix:infinite.banded.matrix}, the operator norms $\norm{\hat{A}}$, $\norm{\hat{A}_N}$, and $\norm{\hat{A}_{22}}$ have the same upper bound $M(b+1)$. Without loss of generality, let us assume that $N\gg b$. For any $\xi \in \ell^2$, if we denote $\xi=\begin{psmallmatrix} \xi_N \\ \tilde{\xi} \end{psmallmatrix}$, then
  \(\delta = \begin{psmallmatrix} \hat{A}_{12}\tilde{\xi} \\ \hat{A}_{21}\xi_N + \hat{A}_{22}\tilde{\xi} \end{psmallmatrix},\)
  so $\norm{\delta}^2 \leqslant \norm{\hat{A}_{12}\tilde{\xi}}^2 + \norm{\hat{A}_{21}\xi_{N}}^2 + \norm{\hat{A}_{22}\tilde{\xi}}^2$. Obviously, $\norm{\hat{A}_{22}\tilde{\xi}}\leqslant M(b+1)\norm{\tilde{\xi}}$. And due to the bandedness of $\hat{A}$, only the entries in the upper right corner of $\hat{A}_{21}$ and the lower left corner of $\hat{A}_{12}$ are nonzero (see \cref{fig:truncation} for illustration). So if we write $\xi=(k_0, k_1, \ldots, k_N, \ldots)'$, and denote $\xi_{N}^{-}=(k_{N-b/2+1}, \ldots, k_N)'$ and $\xi_{N}^{+}=(k_{N+1}, \ldots, k_{N+b/2})'$, then we have $\norm{\hat{A}_{12}\tilde{\xi}}\leqslant M(\tfrac{b}{2}+1)\norm{\xi_{N}^{+}}$ and $\norm{\hat{A}_{21}\xi_{N}}\leqslant M(\tfrac{b}{2}+1)\norm{\xi_{N}^{-}}$. Lastly, we note that $\norm{\tilde{\xi}}$, $\norm{\xi_{N}^{+}}$, and $\norm{\xi_{N}^{-}}$ all go to $0$ as $N\to\infty$, and thus we conclude $\delta\to 0$ as $N\to\infty$. Therefore, by \cref{thm:trotter-kato}, we prove \cref{eq:semigroup.convergence}, where the convergence holds uniformly on any bounded time interval.

  Finally, note that for sufficiently large $N$, $\iota \hat{B}_{N}u = \hat{B}u$,  so for any given $t>0$,
  \[
    \begin{aligned}
      \iota\bar{m}(t) - m(t) = &\bigl(\iota e^{t\hat{A}_{N}}\bar{m}(0) - e^{t\hat{A}}m(0)\bigr) \\
      &+\int_{0}^{t} \bigl(\iota e^{(t-\tau)\hat{A}_N}P_N - e^{(t-\tau)\hat{A}}\bigr) \hat{B}{u}(\tau)\,\mathrm{d}\tau \to 0 \text{ as } N\to\infty.
    \end{aligned}
    \]
\end{proof}

The approximation of the infinite-dimensional linear system by its finite truncations motivates and plays a central role in our development of a control design method for linear ensemble systems, which we will elaborate in the next section. To conclude our investigation of the truncated system, we present some results to another pertinent question that rises naturally with finite truncations, that is, does controllability of the truncated systems carries over to their infinite-dimensional limit?

\begin{proposition}
  \label{prop:controllability.and.truncations}
  The linear ensemble system in \cref{eq:linear_ensemble} with $A(\beta)=\beta A$ and $B(\beta)=B$ is ensemble controllable if and only if all truncated systems of the corresponding moment system are controllable.
\end{proposition}

\begin{proof}
  As shown in \cref{ex:prototype}, for a truncation order $N\in\mathbb{Z}_{+}$, the matrices $\hat{A}_N$ and $\hat{B}_N$ in the truncated dynamics are give by $\hat{A}_{N}=C_{N}\otimes A$ and $\hat{B}_{N}=\sqrt{2}\begin{pmatrix} B & 0 & \ldots & 0 \end{pmatrix}'$, where
  \[
    C_{N}=\begin{pNiceMatrix}
      0 & c_1 & & \\
      c_1 & \ddots & \ddots & \\
       & \ddots & \ddots & c_{N-1} \\
       & & c_{N-1} & 0
    \end{pNiceMatrix}.
  \]
  Accordingly, we obtain the controllability matrix of this finite\hyp{}dimensional system as an upper triangular block matrix:
  \begin{equation}
    \label{eq:controllability.matrix.for.truncated.systems}
    \mathrm{Co}(\hat{A}_{N}, \hat{B}_{N})=\sqrt{2}
    \begin{pNiceArray}{cccc}
      d_0 & & \Block{2-2}<\large>{\star} & \\
      & d_1 & & \\
      & & \smash{\ddots} & \\
      & & & d_{N-1}
    \end{pNiceArray},
  \end{equation}
  where $d_0=B$ and $d_{i}=\bigl(\prod_{j=1}^{i}c_{j}\bigr)A^{i}B$ for $i\geqslant 1$. This ensemble system is $L^2$-ensemble controllable if and only if the span of $\mathrm{Co}(\hat{A},\hat{B})$ is dense in $\ell^2$ by \cref{lem:approx.controllability}. One can verify that if certain truncated moment system is not controllable, i.e., the matrix in \cref{eq:controllability.matrix.for.truncated.systems} is not full rank, there exists a non-trivial left null space for $\mathrm{Co}(\hat{A},\hat{B})$, contradicting  ensemble controllability. On the other hand, controllability of each truncated moment system, together with the upper triangular structure in \cref{eq:controllability.matrix.for.truncated.systems}, indicate that the only vector that left-nullifies $\mathrm{Co}(\hat{A},\hat{B})$ is the trivial zero vector, which proves ensemble controllability.
\end{proof}

\Cref{prop:controllability.and.truncations} is consistent with the results in \cite{li2016ensemble,li2020separating}. In particular, we note that although \cite{li2016ensemble} studies uniform ensemble controllability, it implies  $L^2$-ensemble controllability, because $\|x(T,\cdot)-x_f(\cdot)\|_{L^2}^2 \leqslant K \|x(T,\cdot)-x_f(\cdot)\|_\infty^2$. In general, controllability of an infinite-dimensional linear system and that of its truncations may not coincide, as shown in the following example.

\begin{example}
  \label{ex:truncated.controllable.infinite.noncontrollable}
  Consider the linear system of the form \cref{eq:time-inv.inf.linear.system} where
  \[
    \hat{A}=\begin{pNiceMatrix}[columns-width = 3mm]
      0 & 0 & 0 & \\
      1 & 0 & 0 & \cdots\\
      0 & 1 & 0 & \\
      0 & 0 & 1 & \\
        & & & \smash{\ddots}
    \end{pNiceMatrix} \ \text{and} \ %
    \hat{B}=\begin{pNiceMatrix} 1 \\ -2 \\ 0 \\ 0 \\ \smash{\vdots} \end{pNiceMatrix}.
  \]
  Obviously, for each order $N\geqslant 1$, the truncated system is controllable, since we can verify that its controllability matrix,
  \[
    \mathrm{Co}(\hat{A}_{N},\hat{B}_{N})=
    \begin{pNiceMatrix}
      1 & & & \\
      -2 & \Ddots & & \\
       & \Ddots & & \\
       & & -2 & 1
    \end{pNiceMatrix},
  \]
  is of full rank. However, for the original infinite system, the controllability subspace $\mathrm{Co}(\hat{A},\hat{B})$ is spanned by $\{\sharp^{k}(\hat{B}): k\in\mathbb{N}\}$, where $\sharp$ denotes the forward shift operator. One can check easily that $(1,\tfrac{1}{2},\tfrac{1}{4},\frac{1}{8},\ldots)'$ in $\ell^2$ is orthogonal to $\mathrm{Co}(\hat{A},\hat{B})$. As a result, the infinite\hyp{}dimensional system is not approximately controllable.
\end{example}

\subsection{Control Design through Truncated Moment Systems}
\label{subsec:design}

In \Cref{subsec:truncation}, we defined the finite truncations for the infinite-dimensional linear system in \cref{eq:time-inv.inf.linear.system} and proved that the truncated system converges to the infinite\hyp{}dimensional system as the truncation order $N\to\infty$. This result suggests that the control design of the infinite\hyp{}dimensional system can be carried out on the finite\hyp{}dimensional truncated system if the order is high enough. More specifically, if we denote the initial and target moment states of \cref{eq:time-inv.inf.linear.system} as $m(0)$ and $m_F$, respectively, then for the $N$\textsuperscript{th}-order truncated linear system in \cref{eq:truncated.moment.dynamics} we can apply the classical optimal control design techniques to find a controller $\bar{u}_{N}(t)$ which steers $\bar{m}$ from $P_N m(0)$ at $t=0$ to $P_N m_F$ at $t=T$. Now if we apply the same controller $\bar{u}_{N}(t)$ to the original moment system~\cref{eq:time-inv.inf.linear.system}, then at time $t=T$ we have the moment state $m(T)$ as $m(T)=e^{T\hat{A}}m(0)+\int_{0}^{T} e^{(T-\tau)\hat{A}}\hat{B}\bar{u}_{N}(\tau)\,\mathrm{d}\tau$. By \cref{thm:truncated.systems.approx}, $\iota\bar{m}(T)\to m(T)$ as $N$ increases. Since $\bar{m}(T)=P_N m_F$ also converges to $m_F$ in $\ell^2$ as $N\to\infty$, we conclude that for a sufficiently large truncation order $N$, the controller $\bar{u}_{N}(t)$ will steer the moment state $m(t)$ into a small neighborhood of the target state $m_F$. Therefore, by the duality between the original ensemble and the associated moment system presented in \cref{thm:controllability.equiv.polynomial.case}, we propose a new control design algorithm for linear ensemble systems in \cref{alg:a.priori}, which provides a systematic ensemble control design scheme using truncated moment systems for linear ensemble systems with polynomial coefficients in $\beta$.

\begin{algorithm}[htbp]
  \caption{(A priori) Algorithm for ensemble control design through finite truncations}
  \label{alg:a.priori}
  \begin{algorithmic}[1]
    \Require Initial and target profiles $x_{0}(\cdot)$ and $x_{F}(\cdot)$ in $L^2(K,\mathbb{R}^n)$, $A(\cdot)\in L^\infty(K,\mathbb{R}^{n\times n})$, $B(\cdot)\in L^{2}(K,\mathbb{R}^{n\times m})$, time interval $[0,T]$, error threshold $\epsilon$.
    \Function{ensemble control design}{$x_{0}(\cdot), x_{F}(\cdot)$}
    \State Set a truncation order $N>0$;
    \State $E\gets \epsilon+1$; \Comment{Initialize the error.}
    \While{$E>\epsilon$}
    \State Compute $\hat{A}_{N}$, $\hat{B}_{N}$, $\bar{m}(0)=P_N m(0)$, and $\bar{m}_{F}=P_N m_F$;
    \State For the $N$\textsuperscript{th}-order truncated system~\cref{eq:truncated.moment.dynamics} in $\mathbb{R}^{N}$, find a controller $\bar{u}_{N}$ which steers the state $\bar{m}(t)$ from $\bar{m}(0)$ at $t=0$ to $\bar{m}_F$ at $t=T$;
    \State Apply the controller $\bar{u}_{N}$ to the original linear ensemble in \cref{eq:linear_ensemble}, and compute the state $x(T,\cdot)$;
    \State $E\gets \norm{x(T,\cdot)-x_{F}(\cdot)}_{L^2}$; \Comment{Update the error.}
    \State $N\gets N+1$;
    \EndWhile
    \State \textbf{return} $\bar{u}_{N}$;
    \EndFunction
  \end{algorithmic}
\end{algorithm}

In the following, we present several numerical experiments to demonstrate the applicability and effectiveness of \cref{alg:a.priori}. For all the examples in this section, we design $\bar{u}_{N}(t)$ as the minimal-energy controller for each truncated system. Any other types of controllers for finite-dimensional linear systems will also work.

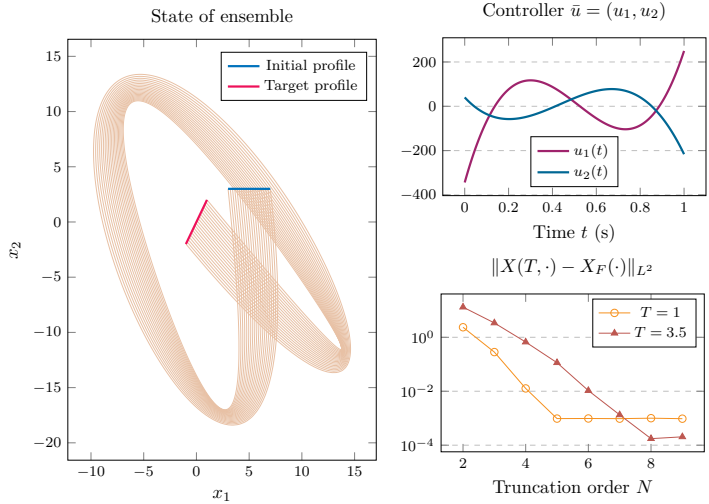
\begin{figure}[htbp]
  \centering
  \begin{subfigure}[b]{0.5\textwidth}
    \raggedleft
    \begin{tikzpicture}[scale=0.7]
      \begin{axis}[
        width=0.9\textwidth,
        height=80mm,
        scale only axis,
        legend style={legend pos=north east, font=\footnotesize},
        grid style=dashed,
        title={State of ensemble},
        xlabel={$x_1$},
        ylabel={$x_2$},
        ylabel near ticks,
        tick label style={font=\footnotesize}
        ]

        \foreach[count=\i, evaluate=\i as \y using int(2*\i-1)]
          \x in {0,2,...,38}{
        \addplot[smooth, color=Tan!60, forget plot]%
          table[x index=\x, y index=\y] {Oscillator_Ensemble/ensemble.dat};
        }
        
        \addplot [
          mark=none,
          very thick,
          color=RoyalBlue,
          domain=3:7,
        ] {3};
          \addlegendentry{Initial profile}

        \addplot [
          mark=none,
          very thick,
          color=OrangeRed,
          domain=-1:1,
        ] {2*x};
          \addlegendentry{Target profile}
      \end{axis}
    \end{tikzpicture}
  \end{subfigure}
  \ %
  \begin{subfigure}[b]{0.48\textwidth}
    \raggedright
    \begin{tikzpicture}[baseline,scale=0.7]
      \begin{axis}[
        width=0.8\textwidth,
        height=30mm,
        scale only axis,
        legend style={at={(axis cs:0.5,-250)}, anchor=center, font=\footnotesize},
        ymajorgrids=true,
        grid style=dashed,
        title={Controller $\bar{u}=(u_1,u_2)$},
        xlabel={Time $t$ (s)},
        tick label style={font=\footnotesize},
        legend entries = {$u_1(t)$, $u_2(t)$}
        ]

        \addplot[smooth, style=very thick, RedViolet]%
          table[x=time, y=u_1] {Oscillator_Ensemble/controller_u.dat};
        \addplot[smooth, style=very thick, MidnightBlue]%
          table[x=time, y=u_2] {Oscillator_Ensemble/controller_u.dat};
      \end{axis}
    \end{tikzpicture}
    \smallskip
    \begin{tikzpicture}[baseline,scale=0.7]
      \begin{semilogyaxis}[
        width=0.8\textwidth,
        height=30mm,
        scale only axis,
        restrict x to domain=2:9,
        legend style={legend pos=north east, font=\footnotesize},
        ymajorgrids=true,
        grid style=dashed,
        title={$\norm{X(T,\cdot)-X_{F}(\cdot)}_{L^2}$},
        xlabel={Truncation order $N$},
        tick label style={font=\footnotesize},
        legend entries={$T=1$, $T=3.5$}
        ]

        \addplot[mark=o, draw=BurntOrange]%
          table[x=N, y=T_1] {Oscillator_Ensemble/error_decay.dat};
        \addplot[mark=triangle*, Maroon!80]%
          table[x=N, y=T_3.5] {Oscillator_Ensemble/error_decay.dat};
      \end{semilogyaxis}
    \end{tikzpicture}
  \end{subfigure}

  \caption{Simulation results for the ensemble of harmonic oscillators in \cref{ex:oscillator.ensemble_1}. The figure on the left shows the trajectories of the ensemble from $t=0$ to $t=1$, the upper right plot is the controller obtained using \cref{alg:a.priori} at $N=5$, and the bottom right plot shows the error~$E$ at each truncation order $N$ for time periods $t\in[0,1]$ and $[0,3.5]$.}
  \label{fig:oscillator.ensemble_1}
\end{figure}
  
\begin{example}[Ensemble of Harmonic Oscillators]
  \label{ex:oscillator.ensemble_1}
  In this example we apply \cref{alg:a.priori} to manipulate an ensemble of harmonic oscillators in $\mathbb{R}^2$, given by $\tfrac{\mathrm{d}}{\mathrm{d}t}X(t,\beta)=\beta AX(t,\beta)+BU(t)$, where $A=\begin{psmallmatrix} 0 & -1 \\ 1 & 0 \end{psmallmatrix}$, $B=I$, and $\beta\in[-1,1]$, from the initial profile $X_{0}(\beta)=(5-2\beta,3)'$ to the target profile $X_{F}(\beta)=(\beta,2\beta)'$ at $T=1$. It is straightforward to see from its associated moment system in \cref{eq:moment.dynamics.prototype.in.example} that each truncated system is controllable, which ensures a minimal-energy controller $\bar{u}_{N}$ in each iteration. \cref{fig:oscillator.ensemble_1} shows sample trajectories of the ensemble, the controller $\bar{u}=(u_1,u_2)$ at $N=5$, and the decay of the error $E=\norm{X(T,\cdot)-X_{F}(\cdot)}_{L^2}$. Observe that in the case of $T=1$, $E$ stops decreasing after $N=5$ as the truncation order goes higher. This is triggered by the ill-conditioning of the matrix in higher order truncated moment systems, which can be resolved if we extend the time period (e.g., $T=3.5$).
\end{example}

\begin{example}[Pattern Formation of Harmonic Oscillators]\label{eg:pattern_design}
  Pattern design is another common yet challenging problem in ensemble control. In this example, we consider pattern formation of the same oscillator ensemble presented in \cref{ex:oscillator.ensemble_1} from a circle-shaped to a square-shaped (dotted gray in \cref{fig:pattern_design}) configuration. The resulting pattern following the ensemble control $(u(t),v(t))$ of duration $T=17$ designed using the truncated moment dynamics of dimension 34 is shown in \cref{fig:pattern_design}. The initial circle is parameterized as $(\cos(\pi\beta),\sin(\pi\beta))$, and the target pattern is expressed as $(4\beta+3,-1)$, $(1,4\beta+1)$, $(-4\beta+1,1)$ and $(-1,-4\beta+3)$ for $\beta\in [-1,-0.5]$, $[-0.5,0]$, $[0,0.5]$ and $[0.5,1]$ respectively, and their corresponding moment terms are derived by computing the integral in \cref{eq:Legendre_moments}, and are displayed in \cref{fig:pattern_design}. 
  The advantage of the proposed method is pronounced when compared with the sampling-based method commonly used for ensemble control problems \cite{chen2014sampling,schonlein2021computation,zlotnik2014optimal}. For the latter method, its performance is usually compromised by the large deviations from the target states for the systems that do not contribute to the control design process \cite{zlotnik2014optimal}. Furthermore, the method in the aforementioned literatures requires way more sampled populations to increase its precision, resulting in larger problem size for computation. In contrast, our sampling-free method guarantees the performance across the subsystems indexed by different parameter $\beta\in [-1,1]$ in the ensemble as indicated by the washed orange color in \cref{fig:pattern_design}.
\end{example}

\begin{figure}[htbp]
  \centering
    \begin{tikzpicture}[scale=0.8]
    \begin{groupplot}[
      group style={
        group size=2 by 1,
        horizontal sep=2mm
    },
      scale only axis,
      height=5cm,
    ]
    \nextgroupplot[width=5cm,title={Pattern formation},
      xlabel={$x_1$},
      ylabel={$x_2$}]
    \addplot graphics [
      xmin=-1.05631,xmax=1.01173,
      ymin=-1.09642,ymax=1.03287,
    ] {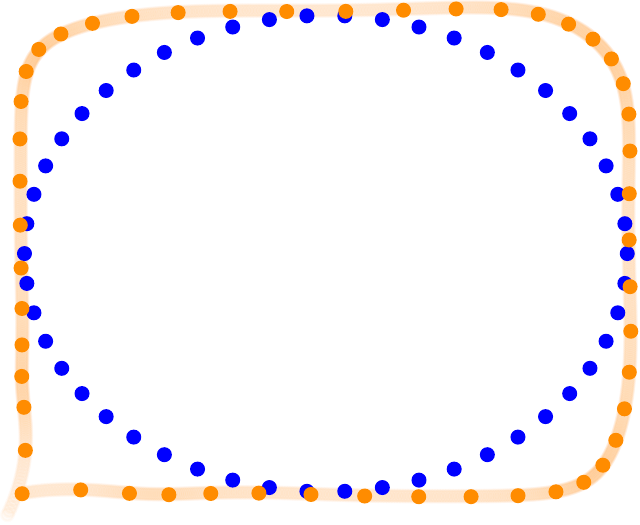};
    \draw[gray,very thick,dashed] (-1,-1) -- (1,-1) -- (1,1) -- (-1,1) -- cycle;

    \nextgroupplot[width=1.5cm,xlabel=Else,
      legend style={at={(0.5,1.02)},anchor=south, font=\footnotesize},
      ymajorgrids=true,
      grid style=dashed,
      ylabel={Coordinates},
      xlabel={Moments},
      ytick = {1,10,20,30,34},
      ylabel near ticks, 
      yticklabel pos=right,
      xtick = {-1,0,0.5},
      tick label style={font=\footnotesize}]
    \addplot[smooth, style=very thick, blue]%
      table[x=X0, y=NN, col sep=comma] {Pattern_Design/Moments.dat};
      \addlegendentry{$\bar{a}_0$}

    \addplot[smooth, style=very thick, myOrange]%
      table[x=Xf, y=NN, col sep=comma] {Pattern_Design/Moments.dat};
      \addlegendentry{$\bar{a}_f$}
    \end{groupplot}
    \end{tikzpicture}
    \medskip
    \begin{tikzpicture}[baseline,scale=0.8]
    \begin{groupplot}[
      group style={
        group size=2 by 1,
        y descriptions at=edge left,
        horizontal sep=2mm
    },
      scale only axis,
      height=1cm,
    ]
    \nextgroupplot[width=3.25cm,ymajorgrids=true,
                grid style=dashed,
                title={Controller $u_1$},
                xlabel={Time $t$ (s)},
                tick label style={font=\footnotesize}]
    \addplot[smooth, style=very thick, blue]%
          table[x=time_downsampling, y=u1, col sep=comma] {Pattern_Design/control_downsample.dat};
    
    \nextgroupplot[width=3.25cm,ymajorgrids=true,
                grid style=dashed,
                title={Controller $u_2$},
                xlabel={Time $t$ (s)},
                tick label style={font=\footnotesize}]
    \addplot[smooth, style=very thick, myOrange]%
          table[x=time_downsampling, y=u2, col sep=comma] {Pattern_Design/control_downsample.dat};
    \end{groupplot}
    \end{tikzpicture}
  \caption{Simulation results of pattern design in \cref{eg:pattern_design}. The dotted blue circle represents the initial profile, and the orange square is the attained final profile. Truncated order $N=17$ is used in this example.}
  \label{fig:pattern_design}    
\end{figure}
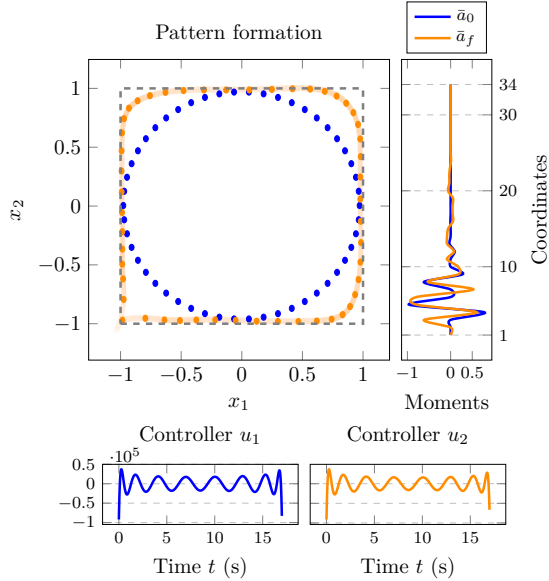

\section{Truncation Error Analysis}
\label{sec:error.analysis.and.sample-free}

In this section we address two issues in implementing \cref{alg:a.priori}, which are closely related to a common problem in ensemble control. First, we notice that in order to compute the norm $\norm{x(T,\cdot)-x_{F}(\cdot)}_{L^2}$, in line~8, one needs to know the value of $x(T,\beta)$ at many sample points for $\beta\in[-1,1]$. Suppose we take, say, 500 sample points of $x(T,\beta)$ uniformly on $\beta\in[-1,1]$, then at each iteration, line~7 needs be executed 500 times for every subsystem. In reality, the size of samples could be much larger, which would severely weaken the efficiency and applicability of the algorithm.

The second issue, which also causes the first one, is considerably more critical, and is ubiquitous in the realm of ensemble control. In general, to examine controllability of an ensemble system, we need to compare the ensemble $x(T,\cdot)$ at time $T$ with the target profile $x_F(\cdot)$, and evaluate the distance
\begin{equation}
  \label{eq:final.and.target.states}
  \norm{x(T,\cdot) - x_F(\cdot)},
\end{equation}
where the norm $\norm{\cdot}$ could be the sup-norm, the $L^p$-norm, etc., depending on the type of controllability in question. In existing literature, the distance~\cref{eq:final.and.target.states} is usually calculated by sampling, e.g., \cite{chen2014sampling,Zeng2017sampled,Zlotnik2012optimal,li2011optimal}. However, it is generally infeasible to numerically compute \cref{eq:final.and.target.states} through discretization with a \emph{quantifiable error}. For the difficulty with the sup\hyp{}norm we refer the readers to \cite{Miao2021numerical}, which discusses the numerical verification of $L^\infty$\hyp{}denseness for ensembles in the Banach space of continuous functions. In our case of $L^2$\hyp{}norms, \cref{eq:final.and.target.states} becomes a numerical integral whose error usually depends on the high order derivatives of the integral function; whereas $x(t,\beta)$ is only assumed to be square\hyp{}integrable in $\beta$. Such an obstacle to obtain an accurate error estimation is in particular deteriorated for studying ensemble systems, because in theory the distance \cref{eq:final.and.target.states} can be made arbitrarily small if the ensemble is controllable. In other words, a small $E$ in line~8 of \cref{alg:a.priori} is no guarantee for reaching the target profile, regardless of the sample size.

\subsection{Estimation of Truncation Error Bounds using Banded Structure}

To avoid the aforementioned issues, instead of computing the distance $\norm{x(T,\cdot)-x_{F}(\cdot)}_{L^2}$ directly, we derive its upper bound which we can evaluate rigorously using moments in $\ell^2$. Henceforth we assume $\hat{A}$ to be a \emph{$b$-banded Hermitian} matrix. Given the initial and target states $m(0)$ and $m_F$, respectively, and suppose we have an admissible controller $\bar{u}(t)$ that steers the $N$\textsuperscript{th}-order truncated moment system from $\bar{m}(0)$ to $\bar{m}(T)=\bar{m}_{F}$, where
\begin{equation}
  \label{eq:control.for.truncated.system}
  \bar{m}(T)=e^{T\hat{A}_N}\bar{m}(0)+\int_{0}^{T} e^{(T-\tau)\hat{A}_{N}} \hat{B}_N\bar{u}(\tau)\,\mathrm{d}\tau.
\end{equation}
Apply the same controller $\bar{u}(t)$ to the original moment system in \cref{eq:time-inv.inf.linear.system}, we get
\begin{equation}
  \label{eq:original.system.using.truncated.control}
  m(T)=e^{T\hat{A}} m(0)+\int_{0}^{T} e^{(T-\tau)\hat{A}} \hat{B}\bar{u}(\tau)\,\mathrm{d}\tau.
\end{equation}
Since $\hat{B}$ has finitely many non-zero entry blocks, $\hat{B}=\iota \hat{B}_N$ for large $N$. In the sequel, by abuse of notation, we will omit the inclusion~$\iota$ when comparing the objects in $\mathbb{R}^{N}$ and $\ell^2$, and let $\norm{\cdot}$ denote the $\ell^2$\hyp{}norm unless otherwise specified. So at time $T$ we have the error term $E=\norm{m(T)-m_F}$ as
\[
  E=\norm{m(T)-m_F} \leqslant \norm{m(T)-\bar{m}_F}+\norm{\bar{m}_{F}-m_{F}}.
\]
By \cref{eq:control.for.truncated.system} and \cref{eq:original.system.using.truncated.control}, we have the upper bound for the first term $\norm{m(T)-\bar{m}_F}$ as
\begingroup
\allowdisplaybreaks
\begin{align}
  \norm{m(T)-\bar{m}_F}&\leqslant\norm{e^{T\hat{A}_N}\bar{m}(0)-e^{T\hat{A}}m(0)} +\norm{\int_{0}^T \Bigl(e^{(T-\tau)\hat{A}}-e^{(T-\tau)\hat{A}_N}\Bigr)\hat{B}\bar{u}(\tau)\,\mathrm{d}\tau} \nonumber \\
  &\leqslant \norm{(e^{T\hat{A}}-e^{T\hat{A}_N})\bar{m}(0)}+\norm{e^{T\hat{A}}(\bar{m}(0)-m(0))} \nonumber\\
  &\qquad\qquad\qquad\qquad\qquad\qquad+\norm{\int_{0}^T \Bigl(e^{(T-\tau)\hat{A}}-e^{(T-\tau)\hat{A}_N}\Bigr)\hat{B}\bar{u}(\tau)\,\mathrm{d}\tau}. \label{eq:a.priori.error.upper.bound}
\end{align}
\endgroup

Observe that for the second term in \cref{eq:a.priori.error.upper.bound}, if $m(0)$ has a closed\hyp{}form that is given, then $\bar{m}(0)-m(0)$ is the remainder of the Fourier\hyp{}Legendre expansion which we can assess accurately. So
\[
  \norm{e^{T\hat{A}}(\bar{m}(0)-m(0))}\leqslant e^{T\norm{\hat{A}}}\cdot \norm{\bar{m}(0)-m(0)}.
\]
Lastly, to determine the first and the third terms in \cref{eq:a.priori.error.upper.bound}, it suffices to compare the corresponding entries in $e^{t\hat{A}}$ and $e^{t\hat{A}_{N}}$. The entrywise estimation of $(e^{t\hat{A}}-e^{t\hat{A}_N})$ is made possible by the band structure of $\hat{A}$. Readers can find a concise summary of \emph{infinite banded matrix} in \cref{appendix:infinite.banded.matrix}. By \cref{lem:bounded.banded.matrix,eq:banded.matrix.norm}, the operator $\norm{\hat{A}}$ is bounded, and we denote its upper bound by $\Delta$. Recall that $\hat{A}$ is Hermitian, and since the spectral radius $|\sigma(\hat{A})|\leqslant \norm{\hat{A}}$, its spectrum $\sigma(\hat{A})$ is contained in $[-\Delta,\Delta]$. Given a truncation order $N$ and the controller $\bar{u}(t)$ in \cref{eq:control.for.truncated.system}, if we partition $\hat{A}$ as
\(\begin{pNiceMatrix}[small]
  \hat{A}_{N} & \hat{A}_{12} \\
  \hat{A}_{21} & \hat{A}_{22}
\end{pNiceMatrix}\) and let
\(\tilde{A}=\begin{pNiceMatrix}[small]
  \hat{A}_{N} & 0 \\
  0 & \hat{A}_{22}
\end{pNiceMatrix}\), then $e^{\hat{A}_{N}}\bar{\xi}=e^{\tilde{A}}\bar{\xi}$ for any $\bar{\xi}\in\mathbb{R}^{N}\subset\ell^2$. To compare the entrywise difference between $e^{t\hat{A}}$ and $e^{t\hat{A}_{N}}$, we partition $e^{t\hat{A}}$ into the following block form
\(e^{t\hat{A}}=\begin{pNiceMatrix}[small]
  S_{11}(t) & S_{12}(t) \\
  S_{21}(t) & S_{22}(t)
\end{pNiceMatrix}\),
where $S_{11}(t)\in\mathbb{R}^{N\times N}$. So for $\bar{\xi}=(\xi_{0},\xi_{1},\ldots,\xi_{N-1})\in\mathbb{R}^{N}\subset\ell^2$,
\begin{equation}
  \label{eq:exp.diff.on.xi.bar}
  (e^{t\hat{A}}-e^{t\hat{A}_{N}})\bar{\xi} =(e^{t\hat{A}}-e^{t\tilde{A}})\bar{\xi}=\begin{pNiceMatrix} [S_{11}(t)-e^{t\hat{A}_{N}}]\bar{\xi} \\ S_{21}(t)\bar{\xi} \end{pNiceMatrix},
\end{equation}
where the lower half in \cref{eq:exp.diff.on.xi.bar} is
\[
  S_{21}(t)\bar{\xi}=\begin{pNiceMatrix}
    \xi_{0}[e^{t\hat{A}}]_{N+1,1}+\ldots+\xi_{N-1}[e^{t\hat{A}}]_{N+1,N} \\
    \xi_{0}[e^{t\hat{A}}]_{N+2,1}+\ldots+\xi_{N-1}[e^{t\hat{A}}]_{N+2,N} \\
    \vdots
  \end{pNiceMatrix}.
\]
By the exponential decay property in \cref{eq:exponential.decay.for.exp.A.hat}, for any $\chi>1$, the $ij$\textsuperscript{th} element $\big|[e^{t\hat{A}}]_{ij}\big|\leqslant K(t)\rho^{|i-j|}$, where $\rho=\chi^{-1}$ and $K(t)=\tfrac{2\chi}{\chi-1}\exp\bigl[\frac{t\Delta(\chi+\chi^{-1})}{2}\bigr]$. Hence, we have
\begin{multline*}
  \big|\xi_{0}[e^{t\hat{A}}]_{N+1,1}+\ldots+\xi_{N-1}[e^{t\hat{A}}]_{N+1,N}\big| \leqslant K(t)(|\xi_{0}|\rho^{N}+|\xi_{1}|\rho^{N-1}+\ldots+|\xi_{N-1}|\rho),
\end{multline*}
and if we set
\begin{equation}
  L_{N}(\bar{\xi})\doteq{}\sum_{k=0}^{N-1} |\xi_{k}|\rho^{N-k},
\end{equation}
then
\begin{align}
  \norm{S_{21}(t)\bar{\xi}} &\leqslant K(t)L_{N}(\bar{\xi})\norm{(1, \rho{}, \rho^{2}, \ldots)'} =\frac{K(t)L_{N}(\bar{\xi})}{\sqrt{1-\rho^2}}. \label{eq:lower.half}
\end{align}
For the upper half in \cref{eq:exp.diff.on.xi.bar}, we note that $[S_{11}(t)-e^{t\hat{A}_{N}}]_{ij}=[e^{t\hat{A}}-e^{t\tilde{A}}]_{ij}$ for any $i,j\in[1,N]$. Therefore, using \cref{eq:exp.perturbation,eq:entrywise.estimation}, we conclude that, for any $i,j\in[1:N]$,
\[
  \big|[e^{t\hat{A}}-e^{t\tilde{A}}]_{ij}\big|\leqslant \bar{K}(t)\rho^{|N-i|+|N-j|-b/2},
\]
where
\[
  \bar{K}(t)=b(b+2)t\norm{\hat{A}}_{\mathrm{max}}\Bigl(\frac{\chi}{\chi-1}\Bigr)^2\exp\Bigl[\frac{t\Delta(\chi+\chi^{-1})}{2}\Bigr],
\]
and $\norm{\hat{A}}_{\mathrm{max}}$ denotes the entrywise maximum of $\hat{A}$. If we define a positive matrix $Q(t)\in\mathbb{R}^{N\times N}$ such that
\[
  [Q(t)]_{ij}=\bar{K}(t)\rho^{|N-i|+|N-j|-b/2},
\]
then
\begin{equation}
  \label{eq:upper.half}
  \norm{(S_{11}(t)-e^{t\hat{A}_N})\bar{\xi}}\leqslant \norm{Q(t)|\bar{\xi}|},
\end{equation}
where $|\bar{\xi}|=(|\xi_{0}|, |\xi_{1}|, \ldots, |\xi_{N-1}|)'$. Combining \cref{eq:lower.half,eq:upper.half}, we obtain the upper bound for \cref{eq:exp.diff.on.xi.bar} as
\begin{equation}
  \label{eq:exp.diff.bound}
  \norm{(e^{t\hat{A}}-e^{t\hat{A}_{N}})\bar{\xi}}\leqslant W(t,\bar{\xi}),
\end{equation}
where
\[
  W^2(t,\bar{\xi})=\norm{Q(t)|\bar{\xi}|}^2+\frac{K^2(t)L_{N}^2(\bar{\xi})}{1-\rho^2}.
\]

Since we have obtained the upper bound for the difference between the exponential function of the infinite banded matrix $\hat{A}$ and its finite truncation $\hat{A}_{N}$, by \cref{eq:exp.diff.bound}, we are now able to estimate the distance between the moment state $m(T)$ and the target $m_{F}$, i.e., $\norm{m(T)-m_{F}}\leqslant E_{N}(T,\bar{u})$, where the truncation error bound is given by
\begin{equation}
  \label{eq:feasible.error.upper.bound}
  \begin{aligned}
    E_{N}(T,\bar{u})\doteq W(&T,\bar{m}(0))+e^{T\norm{\hat{A}}_2}\norm{\bar{m}(0)-m(0)} \\
      &+\int_{0}^{T} W(T-\tau,\hat{B}\bar{u}(\tau))\,\mathrm{d}\tau +\norm{\bar{m}_{F}-m_{F}}.
  \end{aligned}
\end{equation}
From the previous analysis, we know that the error term in \cref{eq:feasible.error.upper.bound} is determined by the bandwidth $b$, the spectrum radius $\Delta$, the norm of $\hat{A}$, the controller $\bar{u}(t)$, the time $T$, and the parameter $\rho$. Note that $\rho$ is a free parameter that can be optimized to achieve a tighter upper bound. In summary, we have the following theorem of error estimation.

\begin{theorem}
  \label{thm:error.uppper.bound}
  Consider a LTI Legendre-moment system of the form \cref{eq:time-inv.inf.linear.system}, where $\hat{A}$ is a Hermitian banded matrix with uniformly bounded entries and $\hat{B}$ is a finite rank operator in the form of \cref{eq:polynomials.in.beta.B.hat}. Given the initial and target states $m(0)$ and $m_F$, respectively, in $\ell^2$, and a truncation order $N$, let $\bar{u}(t)\in L^{\infty}([0,T],\mathbb{R}^m)$ be an admissible controller which steers the $N$\textsuperscript{th}-order truncated moment system~\cref{eq:truncated.moment.dynamics} from $\bar{m}(0)$ to $\bar{m}(T)$ in $\mathbb{R}^{N}$, then the controller $\bar{u}(t)$, when applied to the original Legendre-moment system, would steer the Legendre-moment $m(t)$ from $m(0)$ to $m(T)$ in $\ell^2$ at time $T$, for which we have the following error estimation,
  \[
    \norm{m(T) - m_{F}}_{\ell^2}\leqslant E_{N}(T, \bar{u}),
  \]
  where the upper bound $E_{N}$ is defined in \cref{eq:feasible.error.upper.bound}.
\end{theorem}

Using the error analysis above, we propose the following \emph{sampling-free} ensemble control design protocol described in \cref{alg:sampling-free} for Legendre-moment systems, which are equipped with a Hermitian system matrix. This algorithm substitutes the error term $\norm{x(T,\cdot)-x_{F}(\cdot)}_{L^2}$ in \cref{alg:a.priori} by the upper bound in \cref{eq:feasible.error.upper.bound}. It is worth noting that the evaluation of $E_N$ does not require sampling the ensemble system, which guarantees the uniform performance of the designed ensemble control input over all the individual systems in the ensemble.

\begin{algorithm}[htbp]
  \caption{Sampling-free ensemble control design using moment systems}
  \label{alg:sampling-free}
  \begin{algorithmic}[1]
    \Require Initial profile $x_{0}(\cdot)$ and target profile $x_{F}(\cdot)$ in $L^2(K,\mathbb{R}^n)$, $A(\cdot)\in L^\infty(K,\mathbb{R}^{n\times n})$, $B(\cdot)\in L^{2}(K,\mathbb{R}^{n\times m})$, time interval $[0,T]$, error threshold $\epsilon$.
    \Function{ensemble control design}{$x_{0}(\cdot), x_{F}(\cdot)$}
    \State Set a truncation order $N>0$;
    \State $E\gets \epsilon+1$; 
    \While{$E>\epsilon$}
    \State Compute $\hat{A}_{N}$, $\hat{B}_{N}$, $\bar{m}(0)=P_N m(0)$, and $\bar{m}_{F}=P_N m_F$;
    \State For the $N$\textsuperscript{th}-order truncated system~\cref{eq:truncated.moment.dynamics} in $\mathbb{R}^{N}$, find a controller $\bar{u}_{N}$ which steers the state $\bar{m}(t)$ from $\bar{m}(0)$ at $t=0$ to $\bar{m}_F$ at $t=T$;
    \State Apply the controller $\bar{u}_{N}$ to the original linear ensemble in \cref{eq:linear_ensemble}, and compute the state $x(T,\cdot)$;
      \State $E\gets E_{N}(T,\bar{u}_{N})$; \Comment{Error Upper Bound}
      \State $N\gets N+1$;
    \EndWhile
    \State \textbf{return} $\bar{u}_{N}$;
    \EndFunction
  \end{algorithmic}
\end{algorithm}

\subsection{Simulation Results}
Here, we apply both \cref{alg:a.priori,alg:sampling-free} to the controllable scalar system in \cref{ex:beta.controllable} to illustrate ensemble control design through the Legendre-moment system.

\begin{example}
  \label{ex:scalar.case.control.design}
  Since the banded matrix $\hat{A}$ in \cref{ex:beta.controllable} has a simple structure with bandwidth $2$, we have
  \[
    \big|[e^{t\hat{A}}-e^{t\tilde{A}}]_{ij}\big|\leqslant \bar{K}(t)\rho^{2N-i-j+1},
  \]
  where
  \[
    \bar{K}(t)=2t\norm{\hat{A}}_{\mathrm{max}}\Bigl(\frac{2\chi}{\chi-1}\Bigr)^2\exp\Bigl[\frac{t\Delta(\chi+\chi^{-1})}{2}\Bigr].
  \]
  So we can further improve the positive matrix $Q$ in the upper bound~\cref{eq:exp.diff.bound} to be
  \[
    [Q(t)]_{ij}=\bar{K}(t)\rho^{2N-i-j+1}.
  \]
\end{example}

We compare the performance of both algorithms in \cref{fig:example_sin2cos}, where the initial and target profiles are set to be $x(0,\beta)=\sin(0.5\pi\beta)$ and $x_{F}(\beta)=\cos(0.5\pi\beta)$, respectively. For both algorithms, we use minimal energy control for the finite\hyp{}dimensional truncated systems; and for \cref{alg:sampling-free}, we set $\rho\approx 0.0478$ that minimizes the upper bound~\cref{eq:exp.diff.bound}. Below we highlight some takeaways from the simulation results.

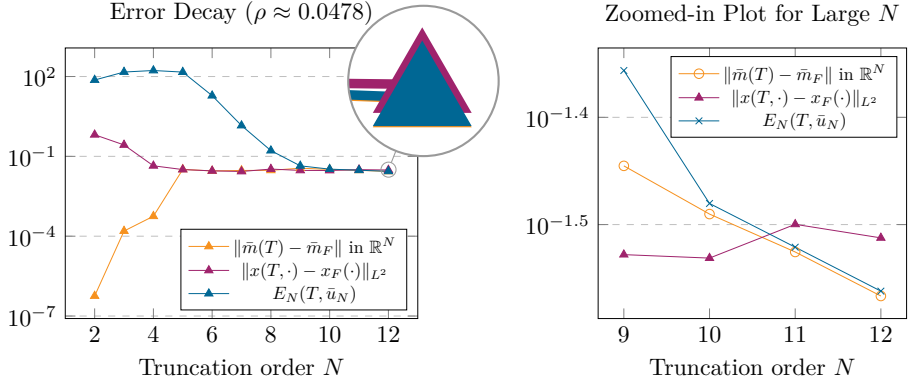
\begin{figure}[htbp]
 \centering
  \begin{tikzpicture}[scale=0.9, spy using outlines={circle, magnification=10, connect spies}]
    \begin{semilogyaxis}[
      width=0.4\textwidth,
      height=40mm,
      scale only axis,
      restrict x to domain=2:12,
      legend style={nodes={scale=0.75, transform shape}, legend pos=south east},
      ymajorgrids=true,
      grid style=dashed,
      title={Error Decay ($\rho\approx 0.0478$)},
      xlabel={Truncation order $N$},
      legend entries = {$\norm{\bar{m}(T)-\bar{m}_{F}}$ in $\mathbb{R}^{N}$, $\norm{x(T,\cdot)-x_{F}(\cdot)}_{L^2}$, $E_{N}{(T,\bar{u}_{N})}$}
      ]

      \addplot[mark=triangle*, BurntOrange]%
        table[x=N, y=finite] {Scalar_sin2cos/Scalar_sin2cos.dat};
      \addplot[mark=triangle*, RedViolet]%
        table[x=N, y=approx] {Scalar_sin2cos/Scalar_sin2cos.dat};
      \addplot[mark=triangle*, MidnightBlue]%
        table[x=N, y=upper_bound] {Scalar_sin2cos/Scalar_sin2cos.dat};

      \coordinate (spypoint) at (axis cs:10.68,0.005);
      \coordinate (magnifyglass) at (axis cs:11.75,5);
    \end{semilogyaxis}

    \spy [black!40, size=2cm] on (spypoint)
      in node[fill=white] at (magnifyglass);
  \end{tikzpicture}
  \ %
  \begin{tikzpicture}[scale=0.9]
    \begin{semilogyaxis}[
      width=0.35\textwidth,
      height=40mm,
      scale only axis,
      restrict x to domain=9:12,
      xtick distance=1,
      title={Zoomed-in Plot for Large $N$},
      legend style={nodes={scale=0.75, transform shape}, legend pos=north east},
      ymajorgrids=true,
      grid style=dashed,
      xlabel={Truncation order $N$},
      legend entries = {$\norm{\bar{m}(T)-\bar{m}_{F}}$ in $\mathbb{R}^{N}$,
      $\norm{x(T,\cdot)-x_{F}(\cdot)}_{L^2}$, $E_{N}{(T,\bar{u}_{N})}$}
      ]

      \addplot[mark=o, draw=BurntOrange]%
        table[x=N, y=finite] {Scalar_sin2cos/Scalar_sin2cos.dat};
      \addplot[mark=triangle*, RedViolet]%
        table[x=N, y=approx] {Scalar_sin2cos/Scalar_sin2cos.dat};
      \addplot[mark=x, MidnightBlue]%
        table[x=N, y=upper_bound] {Scalar_sin2cos/Scalar_sin2cos.dat};
    \end{semilogyaxis}
  \end{tikzpicture}
  \caption{Error decay in the simulation for the scalar linear ensemble in \cref{ex:scalar.case.control.design}, where $x(0,\beta)=\sin(0.5\pi\beta)$ and $x_{F}(\beta)=\cos(0.5\pi\beta)$. The right-side figure is a zoomed-in plot of the left-side figure for $N\geqslant 9$.}
  \label{fig:example_sin2cos}
\end{figure}

\begin{enumerate}[label=(\alph*)]
  \item From the top plot we note that,  as $N$ increases, both the $L^2$\hyp{}distance and its upper bound $E_{N}(T,\bar{u}_{N})$ converge to the error tolerance~$\norm{\bar{m}(T)-\bar{m}_{F}}$ for the truncated moment systems in $\mathbb{R}^{N}$. Namely, the accuracy of both control design methods for the linear ensemble depends on, and is as good as, the accuracy of the method we apply to the finite\hyp{}dimensional truncated moment systems.
  \item As shown in the zoomed-in plot, the computed $L^2$\hyp{}distance is even \emph{larger} than its upper bound~$E_{N}(T,\bar{u}_{N})$ at $N=11$ and $12$. Such discrepancy is the result of sampling for the numerical integral, which, as we mentioned at the beginning of this section, lacks a quantifiable error that we can estimate. In consequence, the $L^2$\hyp{}distance obtained as a numerical integral is not always a reliable indicator of the degree of ensemble controllability. In contrast, the upper bound $E_{N}$ of the $L^2$\hyp{}distance, which is explicitly computable, guarantees that the linear ensemble is steered to within an $\epsilon$\hyp{}neighborhood of the target profile $x_{F}(\cdot)$ using the controller $\bar{u}_{N}$ when $E_{N}(T,\bar{u}_{N})<\epsilon$.
  \item Since \cref{alg:sampling-free} needs \emph{no} sampling, and the controller $\bar{u}_{N}$ is not applied to the subsystems in the linear ensemble at each iteration, it is computationally more efficient than \cref{alg:a.priori}.
\end{enumerate}

\section{Conclusion}
\label{sec:conclusion}

In conclusion, this paper establishes a robust theoretical framework for analyzing and designing control strategies for linear ensemble systems using dynamic Legendre moments. The introduced moment-based approach reveals a duality between the linear ensemble and its moment system, establishing a necessary and sufficient Kalman-type denseness condition for $L^2$-ensemble controllability. This equivalence enables a practical control design scheme through finite truncations in the moment space, leveraging the banded matrix structur of moment dynamics. The demonstrated convergence of truncated moment systems to the Legendre-moment system allows for quantifying error bounds in control design, addressing issues of ill-posedness and inaccuracy associated with conventional sampling processes. Beyond its theoretical contributions, the dynamic Legendre moment method provides a sampling-free, numerically verifiable, and computationally efficient procedure for ensemble control design, offering a valuable tool for precise and reliable control in diverse applications.

\appendix

\section{Shift Operators} \label{appendix:shift.opt}
 
  Given a sequence $\alpha=(\alpha_0, \alpha_1, \ldots)' \in \ell^2(\mathbb{R}^{n})$, where each $\alpha_i$ is a \emph{row} vector of dimension $n$, the forward and the backward shift operator for $\ell^2(\mathbb{R}^{n})$, denoted by $\sharp$ and $\flat$, respectively, are defined as 
  \begin{align*}
    \sharp(\alpha) &= (\underbrace{0, \ldots, 0}_\text{$n$ zeros}, \alpha_0, \alpha_1, \ldots)', \\
    \flat(\alpha) &= (\alpha_1, \alpha_2, \ldots)'.
  \end{align*}

  As we mentioned in \cref{rem:long.remark}~(iii), it is equivalent to shifting forwards or backwards $n$ times using the common shift operators for scalar sequences.

\section{Infinite Banded Matrix} \label{appendix:infinite.banded.matrix}

This appendix provides a self-contained overview of the mathematical theory on (infinite) banded matrix that is essential to the development of the infinite moment system in the paper. The properties of finite and infinite banded matrices and their exponential functions have been extensively studied, see e.g.\ \cite{Benzi2014decay,benzi1999bounds,benzi2007decay,Grimm2012resolvent,Hochbruck2010exponential,shao2014finite}. To be consistent with the previous sections, all infinite matrices in this appendix are one-sided, i.e., their row and column numbers range from $1, 2, \ldots$ to infinity, and the Hilbert space $\ell^2$ consists of one-sided infinite vectors $\xi=[\xi_i], i\in \mathbb{N}_{+}$. For a formal mathematical formulation, let us now consider an infinite matrix $\hat{A}=[a_{ij}]$, $a_{ij}\in\mathbb{R}$ or $\mathbb{C}$ with $i,j\in \mathbb{N}_{+}$.

\begin{definition}
  \label{def:banded.matrix.and.bandwidth}
  The infinite matrix $\hat{A}$ is called $b$-\emph{banded} (\emph{banded} for short) for some even number $b>0$ if $a_{ij}=0$ for all $|i-j|>b/2$. $b$ is called the \emph{bandwidth} of $\hat{A}$.
\end{definition}

It is well-known that a banded matrix is bounded as an operator on $\ell^2$ if all entries in the banded $\hat{A}$ are bounded.

\begin{lemma}
  \label{lem:bounded.banded.matrix}
  $\hat{A}$ is a bounded operator on $\ell^2$ if all its entries $a_{ij}$ are uniformly bounded.
\end{lemma}
\begin{proof}
  Let us say that $|a_{ij}|$ are bounded by some $M>0$, then for any vector $\xi=[\xi_i]\in\ell^2, i\in \mathbb{N}_{+}$, it is evident that
  \begin{multline}
    \label{eq:banded.matrix.norm}
    \norm{\hat{A}\xi}\leqslant M\bigl(\norm{\flat^{\frac{b}{2}}(\xi)}+\norm{\flat^{\frac{b}{2}-1}(\xi)}+\ldots+\norm{\xi}+\norm{\sharp(\xi)} \\
    +\ldots+\norm{\sharp^{\frac{b}{2}}(\xi)}\bigr)\leqslant M(b+1)\norm{\xi},
  \end{multline}
  where $\sharp$ and $\flat$ denote the forward and backward shift operators, respectively.
\end{proof}

Henceforth we shall assume that the infinite banded matrix $\hat{A}$ is bounded. For the remainder of this appendix, we compare the entries in the exponential function of a banded matrix $\hat{A}$ and its finite truncations. Our development modifies and closely follows \cite{shao2014finite}. First, we introduce the concept of exponential decay for infinite matrices.

\begin{definition}
  \label{def:exp.decay}
  An infinite matrix $[a_{ij}]$ has the \emph{exponential decay property} if there exist $K>0$ and $\rho\in(0,1)$ such that
    \begin{equation}
      \label{eq:exp.decay.definition}
      |a_{ij}|\leqslant K\rho^{-|i-j|},
    \end{equation}
  for all $i,j$. The constant $\rho$ is called the \emph{decay rate}.
\end{definition}

The next lemma shows an analytic function of a banded matrix enjoys the exponential decay property.

\begin{lemma}[See~\cite{benzi1999bounds,shao2014finite}]
  \label{lem:benzi-golub}
  Let $D$ be a $b$-banded Hermitian matrix whose spectrum $\sigma(D)$ is contained in $[-1,1]$. Suppose a function $F$ is analytic inside the Bernstein Ellipse $\mathcal{E}_{\chi}$ ($\chi>1$) in $\mathbb{C}$, defined as
  \begin{equation}
    \label{eq:bernstein.ellipse}
    \frac{\mathrm{Re}(z)^2}{(\chi+\chi^{-1})^2} + \frac{\mathrm{Im}(z)^2}{(\chi-\chi^{-1})^2} = \frac{1}{4},
  \end{equation}
  and is continuous on $\mathcal{E}_{\chi}$, and let $M(\chi)\doteq{}\max_{z\in\mathcal{E}_{\chi}}|F(z)|$, then there exists constants $K>0$ and $\rho=\chi^{-2/b}$ such that the entries in the infinite matrix $F(D)$ are of exponential decay:
  \[
    \bigl|[F(D)]_{ij}\bigr|\leqslant K\rho^{|i-j|},
  \]
  where the constant $K=\max\bigl\{\tfrac{2\chi M(\chi)}{\chi -1}, \norm{F(D)}_{2}\bigr\}$.
\end{lemma}

Given a $b$-banded \emph{Hermitian} matrix $\hat{A}$ with $\sigma(\hat{A})\subset [\lambda_{0}-\Delta, \lambda_{0}+\Delta]$, and let $F(z)$ be the entire function $\exp(\Delta z)$ so that \cref{lem:benzi-golub} holds for all $\chi>1$, then for $D=(\hat{A}-\lambda_{0})/\Delta$ with $\sigma(D)\subset [-1,1]$, $F(D)=\exp(\hat{A}-\lambda_{0}I)$. Since $\hat{A}$ commutes with $I$, $\exp(\hat{A}-\lambda_{0}I)=\exp(-\lambda_{0})\exp(A)$. So if we set
\[
  M(\chi)=\max_{z\in\mathcal{E}_{\chi}}|\exp(\Delta z)|=\exp\Bigl[\frac{\Delta(\chi+\chi^{-1})}{2}\Bigr],
\]
then by \cref{lem:benzi-golub}, for any $\chi>1$,
\[
  \bigl|[F(D)]_{ij}\bigr|=\exp(-\lambda_{0})\bigl|[\exp(\hat{A})]_{ij}\bigr|\leqslant K\rho^{|i-j|},
\]
where $\rho=\chi^{-2/b}$ and $K=\max\bigl\{\tfrac{2\chi M(\chi)}{\chi-1}, \norm{F(D)}_{2}\bigr\}$. To determine the value of $K$, first we note that since $F(D)$ is Hermitian, its operator $2$-norm coincides with its spectral radius, then by the spectral mapping theorem,
\[
  \norm{F(D)}_{2}=\max_{x\in \sigma(D)}|F(x)|\leqslant \max_{x\in[-1,1]}|\exp(\Delta x)|=\exp(\Delta).
\]
On the other hand, for any $\chi>1$,
\[
  \frac{2\chi M(\chi)}{\chi-1}=\frac{2\chi}{\chi-1}\exp\Bigl[\frac{\Delta(\chi+\chi^{-1})}{2}\Bigr]>\exp(\Delta).
\]
In conclusion, we have the following exponential decay property for $\exp(\hat{A})$:
\begin{equation}
  \label{eq:exponential.decay.for.exp.A.hat}
  \big|[\exp(\hat{A})]_{ij}\big|\leqslant \exp(\lambda_{0})K\rho^{|i-j|}
\end{equation}
where
$K=\frac{2\chi}{\chi-1}\exp\Bigl[\frac{\Delta(\chi+\chi^{-1})}{2}\Bigr]$ for any $\rho=\chi^{-2/b}\in(0,1)$.

The exponential decay property in \cref{eq:exponential.decay.for.exp.A.hat} allows us to compare $\exp(\hat{A})$ to $\exp(\hat{A}_{N})$. Suppose $\hat{A}$ is partitioned into the block form
\(\hat{A}=
  \begin{pNiceMatrix}[small]
    \hat{A}_{11} & \hat{A}_{12} \\
    \hat{A}_{21} & \hat{A}_{22}
  \end{pNiceMatrix}
\), where $\hat{A}_{11}$ is a finite $N$-by-$N$ matrix, and let $\tilde{A}=\mathrm{diag}\{\hat{A}_{11}, \hat{A}_{22}\}$ and $R=\hat{A}-\tilde{A}=\begin{pNiceMatrix}[small] 0 & \hat{A}_{12} \\ \hat{A}_{21} & 0 \end{pNiceMatrix}$. From the perturbation theory of semigroups, by \cite[Theorem~5.3.1]{Curtain2020} we have that, for any $\varepsilon_{j}$ in the canonical basis of $\ell^2$,
\begin{align}
  [\exp( \hat{A})-\exp(\tilde{A})]\varepsilon_{j} 
    = \int_{0}^{1} \exp[(1-s)\tilde{A}] R \exp(s\hat{A})\varepsilon_{j} \,\mathrm{d}s. \label{eq:exp.perturbation}
\end{align}
Now for each $s\in [0,1]$, denote $U(s)=\exp[(1-s)\tilde{A}]$ and $V(s)=\exp(s\hat{A})$. Due to the banded structure of $R$, we conclude that (assuming that $N>b/2$)
\begingroup
\allowdisplaybreaks
\begin{align}
  \big|[U(s)RV(s)]_{ij}\big| 
    &= \bigg|\sum_{k=-\frac{b}{2}+1}^{0} \sum_{l=1\vphantom{\frac{b}{2}}}^{k+\frac{b}{2}}+\sum_{k=1\vphantom{\frac{b}{2}}}^{\frac{b}{2}} \sum_{l=k-\frac{b}{2}}^{0} U_{i,N+l}R_{N+l,N+k}V_{N+k,j}\bigg| \nonumber \\
    &\leqslant M_{12} \biggl(\sum_{k=-\frac{b}{2}+1}^{0} \sum_{l=1\vphantom{\frac{b}{2}}}^{k+\frac{b}{2}}+\sum_{k=1\vphantom{\frac{b}{2}}}^{\frac{b}{2}} \sum_{l=k-\frac{b}{2}}^{0} |U_{i,N+l}V_{N+k,j}| \biggr) \label{eq:entrywise.estimation},
\end{align}
\endgroup
where $M_{12}$ is the entrywise maximum of $\hat{A}_{12}$ (and $\hat{A}_{21}$). To apply the exponential decay property~\cref{eq:exponential.decay.for.exp.A.hat} to $U(s)$ and $V(s)$, we note that the spectrum $\sigma(\tilde{A})\subseteq [\lambda_{0}-\Delta, \lambda_{0}+\Delta]$, because
\[
  \begin{aligned}
    |\sigma(\tilde{A}-\lambda_{0}I)|&=\norm{\tilde{A}-\lambda_{0}I}_{2} =\max\{\norm{\hat{A}_{11}-\lambda_{0}I}_{2},\ \norm{\hat{A}_{22}-\lambda_{0}I}_{2}\} \\
    &\leqslant \norm{\hat{A}-\lambda_{0}I}_{2}=|\sigma(\hat{A}-\lambda_{0}I)|,
  \end{aligned}
\]
where $|\sigma(\cdot)|$ denotes the spectral radius. Therefore, by \cref{eq:exponential.decay.for.exp.A.hat},
\[
  \begin{aligned}
    |U_{i,N+l}| &\leqslant\frac{2\chi\exp[(1-s)\lambda_{0}]}{\chi-1}\exp\Bigl[\frac{(1-s)\Delta(\chi+\chi^{-1})}{2}\Bigr] \quad \cdot \rho^{|N+l-i|}, \\
    |V_{N+k,j}| &\leqslant\frac{2\chi\exp(s\lambda_{0})}{\chi-1}\exp\Bigl[\frac{s\Delta(\chi+\chi^{-1})}{2}\Bigr]\rho^{|N+k-j|}.
  \end{aligned}
\]
and so we have 
    $ |U_{i,N+l}V_{N+k,j}|\leqslant K'\rho^{|N+l-i|+|N+k-j|}\leqslant K'\rho^{|N-i|+|N-j|-|l-k|}\leqslant K'\rho^{|N-i|+|N-j|-\frac{b}{2}}$,
where $K'=\tfrac{4\chi^2\exp(\lambda_{0})}{(\chi-1)^2}\exp\bigl[\tfrac{\Delta(\chi+\chi^{-1})}{2}\bigr]$. Hence we have
\[
  \big|[U(s)RV(s)]_{ij}\big|\leqslant \frac{b(b+2)}{4}M_{12}K'\rho^{|N-i|+|N-j|-\frac{b}{2}}.
\]
Integrating on $s\in[0,1]$, by \cref{eq:exp.perturbation} we conclude the entrywise truncation error of the exponential function of the Hermitian matrix $\hat{A}$ as
\begin{equation}
  \label{eq:exp.entrywise.truncation.bound}
  \big|[\exp(\hat{A})-\exp(\tilde{A})]_{ij}\big|\leqslant K_{0}\rho^{|N-i|+|N-j|-\frac{b}{2}},
\end{equation}
where $K_{0}=b(b+2)M_{12}\tfrac{\chi^2\exp(\lambda_{0})}{(\chi-1)^2}\exp[\tfrac{\Delta(\chi+\chi^{-1})}{2}]$.

\begin{remark}
  Since our results above hold for all $\rho=\chi^{-2/b}\in (0,1)$, it is sometimes desirable to find a $\chi$ that minimizes $K\rho^{|i-j|}$. Because knowing the minimal upper bound is not necessary in this paper, we will not discuss the optimization of $\chi$, and refer interested readers to \cite{shao2014finite}.
\end{remark}

\section{Semigroup of Bounded Linear Operators} \label{appendix:semigroup}

In this sections, we review the basic concepts and results of semigroup of bounded linear operators, and introduce the Trotter-Kato theorem for the approximation of semigroups, which plays a vital role in the truncation of the moment system in \Cref{subsec:truncation}. A good reference for the material in this section of the appendix is \cite{pazy2012semigroups}.

\begin{definition}
  \label{def:semigroup}
  Let $X$ be a Banach space. A one-parameter family $T(t)$, $0\leqslant t<\infty$, of bounded linear operators in $\mathcal{B}(X)$ is a \emph{semigroup} of bounded linear operators on $X$ if $T(0)=I$ and $T(t+s)=T(t)T(s)$, for any $0\leqslant t, s<\infty$. Furthermore, the semigroup $T(t)$ is called 
  \begin{enumerate}[label=(\alph*)]
    \item \emph{uniformly continuous} if $\lim_{t\to 0^{+}}\norm{T(t)-I}=0$;
    \item \emph{strongly continuous} if $\lim_{t\to 0^{+}}T(t)x=x$ for all $x\in X$, and is denoted as   a $C_0$ semigroup.
  \end{enumerate}

  Given a semigroup $T(t)$ on $X$, its \emph{infinitesimal generator} $A$ is an operator (possibly unbounded) on $X$ given by
  \begin{equation}
    \label{eq:infinitesimal.generator}
    Ax\doteq\lim_{t\to 0^{+}}\frac{T(t)x - x}{t}, \quad x\in D(A),
  \end{equation}
  where the domain $D(A)$ consists of all $x\in X$ such that the limit in \cref{eq:infinitesimal.generator} exists.
\end{definition}

It is straightforward that a uniformly continuous semigroup $T(t)$ is strongly continuous (i.e., of type $C_0$). The following theorem shows every $C_0$ semigroup is exponentially bounded in norm.

\begin{theorem}
  \label{thm:semigroup.norm}
  Let $T(t)$ be a $C_0$ semigroup. There exist constants $\omega\geqslant 0$ and $M\geqslant 1$ such that $\norm{T(t)}\leqslant Me^{\omega t}$ for $0\leqslant t<\infty$. If $A$ is the infinitesimal generator of $T(t)$, then it is a \emph{closed} linear operator and its domain $D(A)$ is \emph{dense} in $X$.

  Further, $T(t)$ is uniformly continuous if and only if its infinitesimal generator $A$ is bounded. Consequently, we have $T(t)=e^{tA}$.
\end{theorem}

In the sequel, we will denote $A\in G(M, \omega)$ for an operator $A$ which is the infinitesimal generator of a $C_0$ semigroup $T(t)$ satisfying $\norm{T(t)}\leqslant Me^{\omega t}$. Next, we present a version of Trotter-Kato theorem for the approximation of semigroups.

\begin{theorem}[see {\cite[pp.~88, Theorem~4.5]{pazy2012semigroups}}]
  \label{thm:trotter-kato}
  Let $A_n\in G(M, \omega)$ and assume (i) as $n\to\infty$, $A_nx\to Ax$ for every $x\in D$ where $D$ is a dense subset of $X$; (ii) there exists a $\lambda_0$ with $\mathrm{Re}\,\lambda_0>\omega$ for which $(\lambda_0I-A)D$ is dense in $X$, then the closure $\bar{A}$ of $A$ is in $G(M,\omega)$. In addition, if $T_n(t)$ and $T(t)$ are the $C_0$ semigroups generated by $A_n$ and $\bar{A}$, respectively, then
  \begin{equation}
    \label{eq:Trotter-Kato}
    \lim_{n\to \infty} T_n(t)x = T(t)x,
  \end{equation}
  for all $t\geqslant 0$. The convergence in \cref{eq:Trotter-Kato} is \emph{uniform} in $t$ for bounded time intervals.
\end{theorem}

\bibliographystyle{siamplain}
\bibliography{legendre_moment}
\end{document}